\documentclass{amsart}

\usepackage{a4wide,amsmath,amssymb}
\usepackage[toc,page]{appendix}

\usepackage{url}
\usepackage{xcolor}

\newtheorem{thm}{Theorem}[section]
\newtheorem{lemma}[thm]{Lemma}
\newtheorem{prop}[thm]{Proposition}
\newtheorem{cor}[thm]{Corollary}
\theoremstyle{remark}
\newtheorem{rem}[thm]{Remark}

\newtheorem{ex}[thm]{Example}
\newtheorem{cex}[thm]{Counterexample}
\theoremstyle{definition}
\newtheorem{defn}[thm]{Definition}
\newtheoremstyle{Claim}{}{}{\itshape}{}{\itshape\bfseries}{:}{ }{#1}
\theoremstyle{Claim}

\newcommand{\R}{\mathbb{R}}

\newcommand{\He}{\mathbb{H}}
\newcommand{\X}{\mathcal{X}}

\newcommand{\iac}{\`i}

\newcommand{\al}{\alpha}

\newcommand{\g}{\gamma}
\newcommand{\Om}{\Omega}

\DeclareMathOperator{\LSC}{LSC}
\DeclareMathOperator{\USC}{USC}
\newcommand{\Sym}{\mathcal{S}}
\theoremstyle{plain}

\def\sideremark#1{\ifvmode\leavevmode\fi\vadjust{
\vbox to0pt{\hbox to 0pt{\hskip\hsize\hskip1em
\vbox{\hsize3cm\tiny\raggedright\pretolerance10000
\noindent #1\hfill}\hss}\vbox to8pt{\vfil}\vss}}}

\begin{document}

\title[Liouville theorems]{Liouville results for fully nonlinear equations modeled on H\"ormander vector fields. I. The Heisenberg group}

\author{Martino Bardi}
\author{Alessandro Goffi} 

\date{\today}
\subjclass[2010]{Primary: 35B53, 35J70, 35J60; Secondary: 49L25, 35H20.}
\keywords{Fully nonlinear equation, degenerate elliptic equation, subelliptic equation, H\"ormander condition, Liouville theorems, Heisenberg 
 group.}
 \thanks{
 The authors are members of the Gruppo Nazionale per l'Analisi Matematica, la Probabilit\`a e le loro Applicazioni (GNAMPA) of the Istituto Nazionale di Alta Matematica (INdAM). The authors were partially supported by the research project ``Nonlinear Partial Differential Equations: Asymptotic Problems and Mean-Field Games" of the Fondazione CaRiPaRo. 
This research was largely carried out while the second-named author was Ph.D. fellow at Gran Sasso Science Institute and the main results are part of his PhD thesis.}
\address{Department of Mathematics ``T. Levi-Civita", University of Padova, Via Trieste 63, 35121 Padova, Italy} \email{bardi@math.unipd.it}
\address{Department of Mathematics ``T. Levi-Civita", University of Padova, Via Trieste 63, 35121 Padova, Italy} \email{alessandro.goffi@math.unipd.it}

\maketitle
\begin{abstract}
This paper studies Liouville properties for viscosity sub- and supersolutions of fully nonlinear degenerate elliptic PDEs, under the main assumption that 
the operator has a family of generalized subunit vector fields that satisfy the H\"ormander condition. A general set of sufficient conditions is given such that all subsolutions bounded above are constant; it includes the existence of a supersolution out of a big ball, that explodes at infinity. Therefore for a large class of operators the problem is reduced to finding such a Lyapunov-like function. This is 
done here for the vector fields that generate the Heisenberg group, giving explicit conditions on the sign and size of the first and zero-th order terms in the equation. 
The optimality of the conditions is shown via several examples. A sequel of this paper applies the methods to other Carnot groups and to Grushin geometries.
\end{abstract}
\tableofcontents

\section{Introduction}
\label{intro}
In this paper we study Liouville properties for viscosity sub- or supersolutions of fully nonlinear degenerate elliptic equations 
\begin{equation}\label{basic}
F(x, u, Du, D^2u)=0 \quad\text{ in }\R^d ,
\end{equation}
where $F : \R^d\times\R\times\R^d\times\mathcal S_d\to \R$ is at least continuous and proper, i.e., nondecreasing in the second entry and non-increasing in the last entry (with respect to the partial order of symmetric matrices). Our main assumption is the existence of a family  
 $\mathcal{X}=(X_1,...,X_m)$ of 
  vector fields satisfying the H\"ormander bracket generating condition and {\em subunit} for $F$ in the following sense: for all $i=1,...,m$
\begin{equation}\label{subunitH}
\sup_{\gamma>0} F(x,0,p,I-\gamma p\otimes p)> 0 \quad \forall p\in \R^{d} \;\text{ such that } \; X_i(x)\cdot p\ne 0 .
\end{equation}
This generalizes the classical definition of subunit vectors for linear operators by Fefferman and Phong \cite{FP} and was introduced in our recent paper \cite{BG}. Typical examples are subelliptic equations of the form 
\begin{equation}\label{1L-}
G(x,u,D
u,(D_{\mathcal{X}}^2u)^*)=0\ ,
\end{equation}
where $(D_{\mathcal{X}}^2 u)_{ij} =X_i (X_j u)$ is the intrinsic (or horizontal) Hessian associated to $\mathcal{X}$,   $Y^*$ is the symmetrized matrix of $Y$, 
and $G$ is proper and strictly decreasing with respect to the last entry. 

\smallskip
Before explaining our results, let us recall some of the many Liouville-type properties  for elliptic equations known  in the literature, the most related to our work.
The classical Liouville theorem for harmonic functions on the whole space states that the only harmonic functions in $\R^d$ bounded from above or below are constants, and it is a consequence of mean-value formulas or, more generally, of the Harnack inequality. Such result actually holds for classical solutions to more general uniformly elliptic equations, provided the zeroth order coefficient has the appropriate sign for the maximum principle, and the equation is homogeneous, see, e.g.,  the monograph \cite{GT}.
For  inhomogeneous equations the property is false, 
for instance $\Delta(|x|^2)=2d$ in $\R^d$. %

The Liouville property 
holds also in the much larger class of merely subharmonic functions 
(i.e., subsolutions of $-\Delta u=0$) bounded from above
if the space dimension is $d=2$,
 by exploiting the behavior of the fundamental solution $\log|x|$ and using the Hadamard Three-Circle Theorem (see, e.g., \cite[Theorem 2.29]{PW}, or 
  Theorem \ref{liolin} below for a different proof). 
  However, this result 
   fails in higher dimensions $d\geq3$: for instance, 
   $u_1(x):=-(1+|x|^2)^{-1/2}$ and $u_2(x):=-(1+|x|^2)^{-1}$ are nonpositive 
    subharmonic functions in $\R^3$ and, respectively, in $\R^d$ with $d\geq4$.  

 For 
 linear degenerate elliptic equations, 
 mean-value properties and Harnack-type inequalities were proved in many cases, 
 typically for vector fields $\X$  that generate a stratified Lie group,  and Liouville theorems for solutions to such equations were proved, e.g., in   \cite{ICDCsub, 
 BLU, KL2, KL1}, see also the references therein. 
On the other hand, one does not expect the Liouville property for sub- or supersolutions to $-\Delta_\X u=0$ when $\X$ 
generates a Carnot group, because the intrinsic dimension of this geometry is larger than 2, which is the maximal one for subharmonic functions.
 In fact, in Section \ref{cex} 
  we give simple explicit examples of bounded, non-constant, classical sub- and supersolutions of the sub-Laplace equation in any Heisenberg group ${\He^d}$.
  
Liouville theorems for nonlinear elliptic equations were first considered by Gidas and Spruck \cite{GS_cpde} for semilinear equations and 
then widely investigated, also in the subelliptic and in the quasilinear settings, see, e.g., { \cite{BCDC, Italosurvey, ICDCsub,  AW, BFP, BMagliaro}} and the references therein. 


For fully nonlinear equations as  \eqref{basic}, in the simpler form $F(x,D^2u)=0$ and uniformly elliptic, it was proved in 
\cite[Section 4.3 Remark 4]{CC} that continuous viscosity solutions either bounded from above or below are constant. We recall that uniform ellipticity with parameters $\Lambda\geq\lambda>0$ can be defined by means of   Pucci's extremal operators (see their definition in Section \ref{puccieqs}) as 
\begin{equation}\label{extremal}
\mathcal{M}^{-}_{\lambda,\Lambda}(M
)\leq F(x,r
,p
,M
) -F(x,r
,p
, 0)\leq \mathcal{M}^{+}_{\lambda,\Lambda}(M
) 
\end{equation}
for all symmetric matrices $M$. The result is a consequence of the Harnack inequality 
and comparison with Pucci's 
 operators.
 Further related results 
 for solutions of PDEs of the form $F(D^2u)=0$ 
 can be found in \cite[Section 1.7]{NV}
 and \cite[Theorem 1.7]{ArmstrongCPAM}, { and in   \cite[Theorem 1.5]{PP} for equations with $F$ depending also on $x$ and $Du$.}
 
 The first results for mere sub- or supersolutions of $F(x,D^2u)=0$ 
  are due to A. Cutr\iac{}  and F. Leoni  \cite{CLeoni}. They proved  that if $u\in C(\R^d)$ is either bounded { above} 
 and satisfying
\begin{equation}\label{M+}
\mathcal{M}^{+}_{\lambda,\Lambda}(D^2u)\leq 0\text{ in }\R^d
\end{equation}
in viscosity sense, or bounded { below} 
 and satisfying 
\begin{equation}\label{M-}
\mathcal{M}^{-}_{\lambda,\Lambda}(D^2u)\geq 0\text{ in }\R^d
\end{equation}
in viscosity sense, then $u$ is constant provided that $d\leq \frac{\Lambda}{\lambda}+1$.
This can be seen as the fully nonlinear analogue of the Liouville theorem for subharmonic functions, since when $\lambda=\Lambda$ one gets the Laplacian (up to constants) and the constraint reads $d\leq 2$. Such conditions are known to be sharp: examples of nontrivial solutions to Pucci's extremal equations when $d>\frac{\Lambda}{\lambda}+1$ can be found in \cite[Remark 2]{CLeoni} and in Section \ref{cex} below.

This result was extended to the Heisenberg group $\He^d$ by Cutr\iac{}  and Tchou \cite[Theorem 5.2]{CTchou} 
for the inequalities \eqref{M+} and  \eqref{M-} with $D^2u$ replaced by $(D^2_{\He^d} u)^*$. Here the condition $d\leq \frac{\Lambda}{\lambda}+1$ is replaced by $Q\leq \frac{\Lambda}{\lambda}+1$, $Q=2d+2$ standing for the homogeneous dimension of $\He^d$. An example of classical subsolution violating 
 the Liouville property when $Q>\frac{\Lambda}{\lambda}+1$ is in Section \ref{cex}. 
This is
consistent with the aforementioned failure of Liouville properties for 
 subharmonic functions 
 in the Heisenberg group. 
 
 In \cite{CLeoni} the authors also prove 
Liouville results for sub- and supersolutions of $F(x,D^2u)+u^p=0$ with $F$ uniformly elliptic, $F(x,0)=0$ and $p$ in a suitable range.
This was recently extended to  Carnot groups of Heisenberg type in \cite{Goffi}. See also \cite{ArmstrongSirakovSharp} and \cite{LeoniJMPA} for related results. 

Liouville properties for PDEs involving gradient terms  of the form 
$$F(x,D^2u)+g(|x|)|Du|+h(x)u^p=0 ,
$$
were first investigated by Capuzzo Dolcetta and  Cutr\iac{} \cite{CDC2003}. 
They assume that $g$ is bounded and such that 
\[
-\frac{\Lambda(d-1)}{|x|}\leq g(|x|)\leq \frac{\lambda-\Lambda(d-1)}{|x|}
\]
for $|x|$ large, and use suitable extensions of the Hadamard three-sphere theorem. 
Note that this is a smallness condition at infinity  on the first order terms of the PDE.
See also \cite{Goffi} for similar recent results on the Heisenberg group. 
Related papers for fully nonlinear PDEs with gradient dependence
are \cite{Rossi}, \cite{PP} and \cite{CFelmer}. 

A new approach to Liouville properties for sub- and supersolutions of 
Hamilton-Jacobi-Bellman elliptic equations involving operators of Ornstein-Uhlenbeck type was introduced 
 in \cite{BCManca}, based on the strong maximum principle and the existence of a sort of Lyapunov function for the equation. 
 It was applied in \cite{BC} to fully nonlinear uniformly elliptic equations of the form \eqref{basic} and to some quasilinear hypoelliptic equations, under 
 assumptions on the sign of the coefficients of the first and zero-th order terms, and on their size. Here these terms must be large for large $|x|$, contrary to the results quoted above.  In the case of Pucci's operators the results of \cite{BC} are different from those in \cite{CLeoni} and fit better the treatment of uniformly elliptic equations via the inequalities \eqref{extremal}.
 In Section \ref{cex} we give examples showing their optimality. 
  The paper \cite{MMT} treats a linear equation on the Heisenberg group in the same spirit.

\smallskip

In the present paper we study Liouville properties for viscosity sub- and supersolutions of equations of the form \eqref{basic} under the condition \eqref{subunitH} for a H\"ormander family $\mathcal{X}$. 
Our main motivation are equations of the form \eqref{1L-}, uniformly subelliptic in the sense that $G$ satisfies the inequalities \eqref{extremal} with Pucci operators $\mathcal{M}^{\pm}_{\lambda,\Lambda}$ acting on $m$-dimensional instead of $d$-dimensional symmetric matrices.
In the first part we prove a general result under two additional assumptions, that for subsolutions are:
a 
 subadditivity condition (\eqref{inf} 
  of Section \ref{abs-res}) which for linear equations corresponds to the homogeneity, and the existence of a Lyapunov-like function $w$ such that $\lim_{|x|\rightarrow\infty}w(x)=+\infty$ and 
supersolution of \eqref{basic}  out of a large ball (symmetric assumptions are made for supersolutions). 
Here we adapt the approach of \cite{BC} to degenerate equations 
by means of the new strong maximum and minimum principles 
obtained by the authors in the recent paper \cite{BG} using 
the generalized subunit vectors  
for fully nonlinear equations \eqref{subunitH}. 

In the second part of the paper we find more explicit sufficient conditions for the Liouville properties in the case of the Heisenberg group $\He^d$ by 
taking  $w=\log \rho$ as  Lyapunov function, 
where   $\rho$ is a norm 1-homogeneous with respect to the dilations of the group $\He^d$. 
As must be expected from the results quoted before, these assumptions 
 concern the sign and the strength of either the first or the zero-th order terms in the equation (or both). { They are related to recurrence conditions in the probabilistic literature, 
  and are a form of dissipativity (cfr., e.g., \cite{PZ}).}
An example of our results, for the uniformly subelliptic equation
\begin{equation}\label{pde_he}
G(x,u,D_{{\mathbb{H}^d}}u,(D_{\mathbb{H}^d}^2u)^*)=0 ,  \quad \text{ in }\R^{2d+1} ,
\end{equation}
where $D_{{\mathbb{H}^d}}u$ and $D_{\mathbb{H}^d}^2u$ are the horizontal gradient and Hessian in $\He^d$, is the following: if 
\[
G(x,r,p,X)\geq \mathcal{M}^-_{\lambda,\Lambda}(X)+\inf_{\alpha\in A}\{c^\alpha(x)r-b^\alpha(x)\cdot p\}
\]
with $c^\alpha\geq 0$, we prove the Liouville property for subsolutions under the condition
\[
\sup_{\alpha\in A}\{b^\alpha(x)\cdot \frac{\eta}{|x_H|^2}-c^\alpha(x)\frac{\rho^4}{|x_H|^2}\log\rho\}\leq \lambda-\Lambda(Q-1) \quad\text{ for } |x|\geq R ,
\]
where $x_H:=(x_1,...,x_{2d})\neq 0_{\R^{2d}}$, $\eta\in\R^{2d}$ 
is defined by
$\eta_i=x_i|x_H|^2+x_{i+d}x_{2d+1}$,
$\eta_{i+d}=x_i|x_H|^2-x_{i}x_{2d+1}$, 
for $i=1,...,d$, 
 and $Q=2d+2$ is the homogeneous dimension of $\He^d$. This condition is satisfied if either $c^\alpha > 0$ or $b^\alpha(x)\cdot {\eta} < 0$ for $x$ large, and under suitable growth conditions at infinity of the data. 
In Section \ref{cex} we use again the  norm $\rho$ to discuss the sharpness of 
 this condition. 

In our companion paper \cite{BG_lio2} we apply the general results of Section \ref{general} to other classical families of H\"ormander vector fields, 
namely, the generators of free step 2 Carnot groups, 
 and Grushin-type fields, whose associated geometry has not a group structure.


It is well known that 
 Liouville properties have many applications to various issues. 
We are motivated in particular by their consequences in ergodic problems, large time stabilization in parabolic equations, and singular perturbation problems, as in, e.g., \cite{BC, BCManca, MMT, MMTesaim}.  
For other forms of Liouville-type theorems for different equations let us also mention \cite{Biri} in the Heisenberg group, the recent paper \cite{Li} for PDEs arising in conformal geometry, and \cite{BGL} for versions of Pucci's extremal equations with 
different 
 degeneracies than in our work. 


\par\smallskip
The paper is organised as follows. In section \ref{glimpse} we explain the approach to Liouville properties based on Lyapunov functions and strong maximum principles in the simple case of classical subsolutions of linear equations, for the reader's convenience, and discuss some related literature. Section \ref{general} presents an abstract result and its various applications to nonlinear equations with general H\"ormander vector fields.
In Section \ref{heisenberg} we study PDEs involving the generators of the Heisenberg group $\He^d$, in the form \eqref{pde_he} where only the horizontal gradient appears, as well as in the form \eqref{1L-} involving the Euclidean gradient. Section \ref{cex} makes a detailed comparison with the literature, in the cases of $\R^d$ and $\He^d$, and discusses by means of explicit examples the optimality of the sufficient conditions for Liouville properties.

\section{A glimpse on the method of proof for 
 linear equations}\label{glimpse} 
  Before showing our main results, we 
  present the proof of a 
   Liouville-type theorem for classical $C^2$ subsolutions to linear uniformly elliptic equations in the Euclidean framework, which serves as a guideline for our proof in the nonlinear and subelliptic setting. 
   It uses only classical arguments such as strong maximum and comparison principles, { but not Harnack inequalities.}
\begin{thm}\label{liolin}
Assume 
 the operator $Lu:=-\mathrm{Tr}(a
(x)D^2u)+b(x)\cdot Du{ +c(x)u}$ is uniformly elliptic, with $a:\R^d\to\mathcal{S}_d$, $b:\R^d\to\R^d$, { $c : \R^d\to [0,+\infty)$ locally bounded}. 
Suppose 
 there exists 
$w\in C^2(\R^d\backslash\{0\})$ such that, for some $R>0$,
\begin{itemize}
\item[(i)] $Lw\geq0$ for $|x|>R$
\item[(ii)] $\lim_{|x|\to+\infty}w=+\infty$.
\end{itemize}
Let also $u\in C^2(\R^d)$ be such that $Lu\leq 0$ 
and ${ 0\leq u}(x)\leq C$ in $\R^d$. Then $u$ is constant. 
\end{thm}
\begin{rem}\label{lapld=2}
This result is essentially  a special case, e.g., of \cite[Theorem 2.1]{BC} and applies 
to the case of the Laplacian (i.e. $a_{ij}=\delta_{ij}$ and $b=0$) when $d\leq 2$; therefore it 
 gives a different proof of the Liouville theorem for subsolutions, see e.g. \cite[Theorem 2.29]{PW}. Indeed, the function $w:=\log|x|$ fulfills the above assumptions, giving thus that every  subharmonic function bounded from above is constant. However, as pointed out in the introduction, this is not the case when $d\geq3$, where 
  $w$ is no longer a classical supersolution of $-\Delta u=0$. { A more general result in the context of Riemannian manifolds 
  can be found in \cite[Corollary 7.7]{Gry1} under the same sufficient conditions (i)-(ii).} It applies to subsolutions of $-\Delta u+b(x)\cdot Du=0$ in any space dimension $d$ under assumptions on the drift $b$ implying the existence of a Lyapunov-like function $w$
  (cfr. \cite{BC}).
\end{rem}
\begin{proof}
For $\zeta>0$ set 
\[
v_\zeta(x):=u(x)-\zeta w(x)\text{ for }|x|\geq\bar{R}
\]
for some $\bar{R}>R>0$. Clearly, $v_\zeta\in C^2(\Omega_{\bar{R}})$, where 
 $\Omega_{\bar{R}}:=\{x\in\R^d:|x|\geq\bar{R}\}$, and 
\[
\text{ 
$Lv_\zeta=Lu-\zeta Lw\leq 0$ for every $x$ such that $|x|>\bar{R}$}\ .
\]
Define $C_\zeta:=\max_{\{|x|=\bar{R}\}}v_\zeta(x)$. Since
\[
\lim_{|x|\to+\infty}v_\zeta(x)=-\infty\ ,
\]
there exists $K_\zeta>\bar{R}$ such that $v_\zeta<C_\zeta$ for every $x$ such that $|x|\geq K_\zeta$. By the weak maximum principle (see \cite[Corollary 3.2]{GT}) 
 on the set $\{x\in\R^d:\bar{R}<|x|<K_\zeta\}$ we have
\[
\max_{\{x\in\R^d:\bar{R}<|x|<K_\zeta\}}v_\zeta(x)=\max_{\{x\in\R^d:|x|=\bar{R}\text{ or }|x|=K_\zeta\}}v_\zeta(x) = C_\zeta .
\]
Since $v_\zeta(x)<C_\zeta$ for every $x$ such that $|x|\geq K_\zeta$, we { get, for all $|y|\geq\bar R$},
\[
v_\zeta(y) 
\leq C_\zeta 
\leq  \max_{\{x\in\R^d:|x|=\bar{R}\}}u-\zeta \min_{\{x\in\R^d:|x|=\bar{R}\}}w\ .
\]
On one hand, letting $\zeta\to0$ we conclude
\[
u(y)\leq  \max_{\{x\in\R^d:|x|=\bar{R}\}}u , \quad \text{ for all $|y|\geq\bar{R}$}\ .
\]
On the other hand, owing to the weak maximum principle in the set $B(0,\bar{R})$ we obtain
\[
u(y)\leq  \max_{\{x\in\R^d:|x|=\bar{R}\}}u  , \quad\text{ for all $|y|<\bar{R}$}\ .
\]
Combining the above inequalities one concludes
\[
u(y)\leq \max_{\{x\in\R^d:|x|=\bar{R}\}}u, \quad \text{ for all }y\in\R^d\ .
\]
Hence, $u$ attains its { nonnegative} maximum 
at some point of  $\partial B(0,\bar{R})$ and then the conclusion follows by the strong maximum principle for classical linear uniformly elliptic equations \cite[Theorem 3.5]{GT}.
\end{proof}
\begin{rem}
The same result remains true if $L$ is replaced by a degenerate elliptic operator $L_\X u:=-\sum_{i,j}X_iX_ju+b(x)\cdot D_\X u{ +c(x)u}$, provided the vector fields  $\X$ satisfy the H\"ormander condition and $b:\R^m\to\R^n$, $n\leq m$ is { smooth}, 
 the proof being exactly the same thanks to Bony strong maximum principle for subelliptic equations. 
An example of such result is  \cite[Proposition 3.1]{MMT}. 

Note also that the assumption $u\leq C$ can be replaced by $\limsup_{|x|\to\infty}u(x)/w(x)\leq 0$, whereas the sign condition $u\geq 0$ can be dropped if $c\equiv 0$.
 
We further remark that, when $b\equiv0$, 
$L$ reduces to a Schr\"odinger-type operator. When $u$ is a solution of the equation $-\Delta u+c u=0$, the Liouville property is proved in \cite[Corollary 13.7]{Gry1} under the same sufficient conditions (i)-(ii) on Riemannian manifolds, and it is connected to recurrence and non-explosive properties of Brownian motions on Riemannian manifolds, see \cite[Theorem 5.1 and Section 13.2]{Gry1}. Our result is more general in that it allows $u$ to be merely a subsolution to the equation. 
 We also refer to \cite{PZ} for a control theoretic interpretation of the Liouville property { for Ornstein-Uhlenbeck operators}.
\end{rem}

\section{The general case}\label{general}
\subsection{An abstract result} \label{abs-res}

In this Section we consider a general equation of the form
\begin{equation}\label{fully2}
F(x,u,Du,D^2u)=0 \quad \text{ in }\R^d .
\end{equation}
We will denote $F[u]:=F(x,u,Du,D^2u)$ and make the following assumptions
\begin{itemize}
  \item[(i)] { $F : \R^d\times\R\times\R^d\times\mathcal S_d\to \R$} is continuous, proper, satisfies
  \begin{equation}\label{inf}\tag{S1}
 F[\varphi-\psi]\leq F[\varphi]-F[\psi]\text{ for all }\varphi,\psi\in C^2(\R^d)
  \end{equation}
  and $F(x,r,0,0)\geq0$ for every $x\in\Omega$ and $r\geq0$.
  
  \item[(ii)] $F$ satisfies the comparison principle in any bounded open set $\Om$, namely, if $u$ and $v$ are, respectively, a viscosity sub- 
  and  supersolution of \eqref{fully2} such that $u\leq v$ on $\partial\Om$, then $u\leq v$ in $\Om$.  
\item[(iii)] There exists $R_o\geq 0$ and $w\in\LSC(\R^d)$ viscosity supersolution of \eqref{fully2} for $|x|>R_o$ and  satisfying $\lim_{|x|\rightarrow\infty}w(x)=+\infty$.
  \item[(iv)] $F$ satisfies the strong maximum principle, namely, any viscosity subsolution of \eqref{fully2} that attains an 
   non-negative maximum must be constant.

\end{itemize}
To prove the analogous results for viscosity supersolutions we need to replace (i) and (iii)-(iv) above by
\begin{itemize}
  \item[(i')] $F$ is continuous, proper, satisfies
   \begin{equation}\label{sup}\tag{S2}
 F[\varphi-\psi]\geq F[\varphi]-F[\psi]\text{ for all }\varphi,\psi\in C^2(\R^d)
  \end{equation}
  and $F(x,r,0,0)\leq0$ for every $x\in\Omega$ and $r\leq0$.
 
  \item[(iii')] There exists $R_o\geq0$ and $W\in\USC(\R^d)$ viscosity subsolution to \eqref{fully2} for $|x|>R_o$ and satisfying $\lim_{|x|\rightarrow\infty}W(x)=-\infty$. 
    \item[(iv')] $F$ satisfies the strong minimum principle.
\end{itemize}
The next result extends the proof 
of Theorem \ref{liolin} for linear equations to the fully nonlinear degenerate setting. Its proof is essentially the same as the one done  in \cite{BC} for HJB 
 equations,
 so we only outline it for the reader's convenience.
\begin{prop}\label{liogen1}
Assume {\upshape (i)-(iv)}. Let $u\in\USC(\R^d)$ be a viscosity subsolution to \eqref{fully2} satisfying 
\begin{equation}\label{limsup}
\limsup_{|x|\rightarrow\infty}\frac{u(x)}{w(x)}\leq0\ .
\end{equation}
for $w$ as in {\upshape(iii)}. If $u\geq0$, then $u$ is constant.
\end{prop}

\begin{proof} 
Define $u_{\zeta}(x):=u(x)-\zeta w(x)$ for $\zeta>0$. Possibly increasing $R_o$, we can assume that $u$ is not constant in $\overline{B}(0,R_o):=\{x\in\R^m:|x|\leq R_0\}$, otherwise we are done. Set
\begin{equation}
\label{Ceta}
C_{\zeta}:=\max_{|x|\leq R_o}u_{\zeta}(x)\ .
\end{equation}
Note that $F[C_{\zeta}]\geq0$ and $C_{\zeta}>0$ for all $\zeta$ sufficiently small. 
In fact, if $C_{\zeta}=0$, 
by letting $\zeta\rightarrow0$ we get $u(x)=0$ for every $x$ with $|x|\leq R_o$, a contradiction with 
$u$ not constant in $\overline{B}(0,R_o)$.

 The growth condition \eqref{limsup} implies
\begin{equation*}
\limsup_{|x|\rightarrow\infty}\frac{u_{\zeta}(x)}{w(x)}\leq-\zeta<0 \quad \forall \zeta>0\ .
\end{equation*}
As a consequence, we have 
\begin{equation}\label{limsupueta}
\lim_{|x|\rightarrow+\infty}u_{\zeta}(x)=-\infty\ .
\end{equation}
Then, for all $\zeta>0$ there exists $R_{\zeta}>R_o$ such that
\begin{equation}\label{ineqC}
u_{\zeta}(x)\leq C_{\zeta}\ \quad\text{for all }|x|\geq R_{\zeta}\ .
\end{equation}

The main step is proving that $u_{\zeta}$ is a viscosity subsolution of $F[u]=0$ in $\{x\in\R^d:|x|>R_o\}$. 
Take $\bar{x}$ and $\varphi$ smooth such that $0=(u_{\zeta}-\varphi)(\bar{x})>(u_{\zeta}-\varphi)({x})$ for all $x$. 
Assume by contradiction that $F[\varphi(\bar{x})]>0$. 
Then for some $\delta>0$ and $0<r<|\bar{x}|-R_o$ 
\begin{equation}
\label{F>0}
F[\varphi-\delta]>0  \quad \text{ in }B(\bar{x},r) .
\end{equation}
Next take $0<k<\delta$ such that $u_{\zeta}-\varphi\leq-k<0$ on $\partial B(\bar{x},r)$. 
We claim that $\zeta w+\varphi-k$  satisfies $F[\zeta w+\varphi-k]\geq0$ in $B(\bar{x},r)$. Indeed, take $\tilde{x}\in B(\bar{x},r)$ and $\psi$ smooth such that
$\zeta w +\varphi-k-\psi$ has a minimum at $\tilde{x}$.
Using that $w$ is a viscosity supersolution to \eqref{fully2}, 
$F$ proper, \eqref{inf}, and  \eqref{F>0} we get
\[
0\leq F[\psi(\tilde{x})-\varphi(\tilde{x})+k]\leq F[\psi(\tilde{x})-\varphi(\tilde{x})+\delta]\leq F[\psi(\tilde{x})]-F[\varphi(\tilde{x})-\delta]<F[\psi(\tilde{x})] .
\]
Then $\zeta w+\varphi-k$ is a 
supersolution to $F[u]=0$ in $B(\bar{x},r)$ and $u\leq\zeta w+\varphi-k$ on $\partial B(\bar{x},r)$, 
so  the comparison principle gives 
$u\leq \zeta w+\varphi-k$ in $B(\bar{x},r)$,
in contradiction with the fact that $u(\bar{x})=\zeta w(\bar{x})+\varphi(\bar{x})$.

Now we can use the comparison principle in $\Om=\{x:R_o<|x|<R_{\zeta}\}$ and \eqref{ineqC} to get 
$u_{\zeta}\leq C_{\zeta}$ in $\Om$. 
 Therefore we have
\begin{equation*}
u_{\zeta}(x)\leq C_{\zeta}\  \quad \text{for all }|x|\geq R_o .
\end{equation*}
By letting $\zeta\rightarrow0^+$ we obtain 
\begin{equation*}
u(x)\leq\max_{|y|\leq R_o}u(y)  \quad \text{for all } x\in \R^d ,
\end{equation*}
and hence $u$ attains its maximum $\bar{x}$ over $\R^d$. Since 
 $u\geq0$ the SMP gives 
  the desired conclusion.
\end{proof}
\begin{rem}
Note that if $u$ is bounded above, then \eqref{limsup} is satisfied. 
\end{rem}
The next result says 
 that the assumption $u\geq0$ can be dropped provided $r\mapsto F(x,r,p,X)$ 
is constant: this will be the case for some HJB operators we discuss in the next sections.
\begin{cor}\label{corlio1}
Assume {\upshape (i)-(iv)}. Let $u\in\USC(\R^d)$ be a viscosity subsolution to \eqref{fully2} satisfying \eqref{limsup}
for $w$ as in {\upshape(iii)}. Assume  $r\mapsto F(x,r,p,X)$ is constant for all $x,p, X$ and $F(x,r,0,0)=0$ for every $x\in\Omega$. 
Then $u$ is constant.
\end{cor}
\begin{proof}   
The proof goes along the same lines as Proposition \ref{liogen1}. It is sufficient to note that since $r\mapsto F(x,r,p,X)$ is constant for all $x,p, X$, $u+|u(\bar{x})|$, $\bar x$ standing for the maximum point in Proposition \ref{liogen1}, is again a subsolution, 
and one concludes.
\end{proof}

A symmetric result holds for the case of supersolutions to \eqref{fully2}, { see \cite{BC} for the details of the proof.}
\begin{prop}\label{liogen2}
Assume {\upshape (i'),(ii),(iii') and (iv')}. Let $v\in\LSC(\R^d)$ be a viscosity supersolution to \eqref{fully2} satisfying 
\begin{equation}\label{liminf}
{ \limsup_{|x|\rightarrow\infty}}\frac{v(x)}{W(x)}\leq0
\end{equation}
for $W$ as in {\upshape(iii')}. Assume either $v\leq0$, or $r\mapsto F(x,r,p,X)$ is constant for all $x,p, X$ and $F(x,r,0,0)=0$ for every $x\in\Omega$. Then $v$ is constant.
\end{prop}

\subsection{Equations with H\"ormander vector fields}
In this section we discuss Liouville properties for PDEs over H\"ormander vector fields. We recall that 
the vector fields 
{ $Z_1,..., Z_m$ 
satisfy the H\"ormander's rank condition if

\smallskip
\noindent (H) \emph {the vector fields are smooth 
and the 
 Lie algebra generated by them has full rank $d$ at each point.}
 \smallskip
 
 \noindent
  The classical smoothness requirement on $Z_{i}$ is $C^\infty$,} but it can be reduced to $C^k$ for a suitable $k$, and considerably more if the Lie brackets are interpreted in a generalized sense, see \cite{FR} and the references therein.

Before stating the main result for subsolutions, we recall a crucial scaling assumption for the validity of the strong maximum principle for fully nonlinear subelliptic equations together with the concept of generalized subunit vector field.
\begin{itemize}
\item[(SC)] For some $\phi : (0,1]\to (0,+\infty]$, $F$ 
satisfies
\begin{equation*}
F(x,\xi s,\xi p,\xi X)\geq\ \phi(\xi)F(x,s,p,X)
\end{equation*}
for all $\xi\in(0,1]$, 
 $s\in[-1,0]$, 
  $x\in\Omega$, $p\in\R^d\backslash\{0\}$, and $X\in\Sym_d$;
\end{itemize}

\begin{defn}\label{subunit}  $Z\in \R^{d}$ is a generalized \emph{subunit vector} (briefly, SV)
 for $F=F(x,r,p,X)$ at $x\in\Omega$ if 
\begin{equation*}
\sup_{\gamma>0} F(x,0,p,I-\gamma p\otimes p)> 0 \quad \forall p\in \R^{d} \;\text{ such that } \; Z\cdot p\ne 0 ;
\end{equation*}
$Z:\Omega\rightarrow\R^{d}$ is a \emph{subunit vector field} (briefly, SVF)
if $Z(x)$ is SV for $F$ at $x$ for every $x\in\Omega$. 
\end{defn}
The name is motivated by the the notion introduced  by C. Fefferman and D.H. Phong \cite {FP} for linear operators. 
\begin{thm}\label{liosub}
Let $F$ be such that {\upshape (i),(ii), (iii), and (SC)} hold. Furthermore assume that $F$ admits $Z_1,...,Z_m$ generalized subunit vector fields satisfying the H\"ormander condition {\upshape (H)}. 
 Let $u\in\USC(\R^d)$ be a viscosity subsolution to \eqref{fully2} satisfying \eqref{limsup} for $w$ as in {\upshape(iii)}. Assume either $u\geq0$, or $r\mapsto F(x,r,p,X)$ is constant for all $x,p, X$ and $F(x,r,0,0)=0$ for every $x\in\Omega$. 
 Then $u$ is constant.
\end{thm}
\begin{proof}
The proof is a consequence of Proposition \ref{liogen1} and Corollary \ref{corlio1}, recalling that under {\upshape(i), (SC)} and the existence of subunit vector fields for $F$, the strong maximum principle holds (cf \cite[Corollary 2.6]{BG}).
\end{proof}
Similarly, in the case of supersolutions we have the following result by replacing (SC) with
\begin{itemize}
\item[(SC')] For some $\phi : (0,1]\to (0,+\infty]$, $F$ 
satisfies
\begin{equation*}
F(x,\xi s,\xi p,\xi X)\leq\ \phi(\xi)F(x,s,p,X)
\end{equation*}
for all $\xi\in(0,1]$, 
 $s\in[-1,0]$, 
  $x\in\Omega$, $p\in\R^d\backslash\{0\}$, and $X\in\Sym_d$;
\end{itemize}
and the condition in Definition \ref{subunit} is replaced with
\begin{equation*}
\inf_{\gamma>0} F(x,0,p,\gamma p\otimes p-I)> 0 \quad \forall p\in \R^{d} \;\text{ such that } \; Z\cdot p\ne 0 .
\end{equation*}
{ The proof is a simple consequence of 
 Proposition \ref{liogen2}, 
 and the strong minimum principle Corollary 2.12 in 
 \cite
 {BG}.
}
\begin{thm}\label{liosuper}
Let $F$ be such that {\upshape (i'),(ii),(iii'), and (SC')} hold. Furthermore, assume that $F$ admits $Z_1,...,Z_m$ generalized subunit vector fields satisfying the H\"ormander condition as above. Let $v\in\LSC(\R^d)$ be a viscosity supersolution to \eqref{fully2} satisfying \eqref{liminf} for $W$ as in {\upshape(iii')}. Assume either $v\leq0$, or $r\mapsto F(x,r,p,X)$ is constant for all $x,p, X$ and $F(x,r,0,0)=0$ for every $x\in\Omega$. Then $v$ is constant.
\end{thm}

{ Next we apply the last two theorems to subelliptic equations of the form 
\begin{equation}
\label{1Lbis}
G(x,u,D_{\mathcal{X}}u,(D_{\mathcal{X}}^2u)^*)=0 ,  \quad \text{ in }\R^d ,
\end{equation}
where $\mathcal{X}=(X_1,...,X_m)$ are $C^{1,1}$ vector fields 
on $\R^d$ satisfying the H\"ormander condition (H), $D_{\mathcal{X}}u:=(X_1 u,...,X_m u)$, $(D_{\mathcal{X}}^2 u)_{ij} :=X_i (X_j u)$, and $Y^*$ is the symmetrized matrix of $Y$. 
Here { $G : \R^d\times\R\times\R^m\times\mathcal S_m\to\R$} is at least continuous, proper, satisfying \eqref{inf},  and it is \emph{elliptic for any $x$ and $p$ fixed} 
 in the following sense:
\begin{equation}
\label{Gell}
\sup_{\gamma>0} G(x,0,q,X-\gamma q\otimes q)> 0 \quad \forall \, x\in \Om,\; q\in \R^{m}, \; q\ne 0 , \; X\in \Sym_m . 
\end{equation}
After choosing a basis in Euclidean space we write $X_j=\sigma^j\cdot D$, with $\sigma^j:\R^d\rightarrow\R^d$, and $\sigma=\sigma(x)=[\sigma^1(x),...,\sigma^m(x)]\in\R^{d\times m}$. 
Then
\begin{equation*}
D_{\mathcal{X}}u=\sigma^T Du=(\sigma^1\cdot Du,...,\sigma^m\cdot Du)
\end{equation*}
and
\begin{equation*}
X_i (X_j u)=(\sigma^T D^2 u\ \sigma)_{ij} +(D\sigma^j\ \sigma^i)\cdot Du\ .
\end{equation*}
Therefore,  for $u\in C^2$,
\begin{equation*}
(D_{\mathcal{X}}^2 u)^*=\sigma^T D^2 u\sigma+
g(x,Du) \,, \qquad (g(x,P))_{ij}:=\frac12[(D\sigma^j\ \sigma^i)\cdot p+(D\sigma^i\ \sigma^j)\cdot p] ,
\end{equation*}
%
and we can  rewrite the equation \eqref{1Lbis} in Euclidean coordinates, i.e., in the form $F(x,u,Du,D^2 u)=0$, by taking
\begin{equation}\label{FG}
F(x,r,p,X)=G(x,r,\sigma^T(x)p,\sigma^T(x)X\sigma(x)+g(x,Du)) .
\end{equation}
The ellipticity condition \eqref{Gell} implies that 
the vector fields  $X_j=\sigma^j\cdot D$ are subunit for $F$  (cfr. \cite[Lemma 3.1]{BG}). 
Moreover, if $G$ satisfies (SC) or (SC'), also $F$ does. Assume finally that $F$ given by \eqref{FG} satisfies property (ii) about the weak comparison principle. Then we have the following.
\begin{cor}
Under the assumptions listed above the equation \eqref{1Lbis} has the Liouville properties for viscosity sub- and supersolutions of Theorems \ref{liosub} and 
 \ref{liosuper} .
\end{cor}
}

\subsection{Equations driven by Pucci's subelliptic operators}
\label{puccieqs}


{ Given a family of $m$ vector fields and 
 the corresponding Hessian matrix, we consider the Pucci's extremal operators over such matrices instead of the classical Euclidean Hessians. 
Following Caffarelli and Cabr\'e \cite{CC}, } we fix $0<\lambda\leq\Lambda$, denote with $\Sym_{m}$ the set of $m\times m$ symmetric matrices, and let
\[
\mathcal{A}_{\lambda,\Lambda}:=\{A\in\Sym_m
:\lambda|\xi|^2\leq A\xi\cdot\xi\leq\Lambda|\xi|^2\ ,\forall\xi\in\R^m
\}\ .
\]
{ For  $M\in\Sym_{m}$ the \textit{maximal} and \textit{minimal 
 operator} are defined as
\begin{equation*}
\mathcal{M}^{+}_{\lambda,\Lambda}(M):=\sup_{A\in \mathcal{A}_{\lambda,\Lambda}}\mathrm{Tr}(-AM)
 , \qquad \mathcal{M}^{-}_{\lambda,\Lambda}(M):=\inf_{A\in\mathcal{A}_{\lambda,\Lambda}}\mathrm{Tr}(-AM)\ .
\end{equation*}
}
If $e_1\leq...\leq e_d$ are the ordered eigenvalues of the matrix $M$, one can check  that \cite[Section 2.2]{CC}
\begin{equation}
\label{representationM+}
\mathcal{M}^{+}_{\lambda,\Lambda}(M)=-\Lambda\sum_{e_k<0}e_k-\lambda\sum_{e_k>0}e_k , \qquad 
\mathcal{M}^{-}_{\lambda,\Lambda}(M)=-\Lambda\sum_{e_k>0}e_k-\lambda\sum_{e_k<0}e_k .
\end{equation}

Now we prove the Liouville property for subsolutions of the equation
\begin{equation}\label{pucci1}
\mathcal{M}^-_{\lambda,\Lambda}((D^2_{\mathcal{X}}u)^*)+H_i(x,u,D_{\mathcal{X}}u)=0 \quad \text{ in }\R^d\ ,
\end{equation}
where
\begin{equation}\label{Hi}
H_i(x,r,p):=\inf_{\alpha\in A}\{c^\alpha(x)r-b^\alpha(x)\cdot p\}
\end{equation}
and for 
 supersolutions of
\begin{equation}\label{pucci2}
\mathcal{M}^+_{\lambda,\Lambda}((D^2_{\mathcal{X}}u)^*)+H_s(x,u,D_{\mathcal{X}}u)=0 \quad \text{ in }\R^d\ ,
\end{equation}
where
\begin{equation}\label{Hs}
H_s(x,r,p):=\sup_{\alpha\in A}\{c^\alpha(x)r-b^\alpha(x)\cdot p\}\ .
\end{equation}
{ Here $A$ is a set of indices such that $H_i$ and $H_s$ are finite.} We assume that { $b^\alpha : \R^d \to \R^m$ } 
 is locally Lipschitz in $x$ uniformly in $\alpha$, i.e., for all $R>0$ there exists $K_R>0$ such that
\begin{equation}\label{blip}
\sup_{|x|,|y|\leq R, \alpha\in A}|b^\alpha(x)-b^\alpha(y)|\leq K_R|x-y|
\end{equation}
and 
\begin{equation}\label{cass}
c^\alpha(x)\geq0\text{ and continuous in $|x|\leq R$ uniformly in $\alpha$}.
\end{equation}
\begin{cor}\label{pucciminmax} { Assume the vector fields $\mathcal{X}$ 
are $C^{1,1}$ 
and satisfy the H\"ormander condition (H).}
\begin{itemize}
\item [(a)] Under the previous assumptions 
 on $H_i$, let $u\in\USC(\R^d)$ be a viscosity subsolution to \eqref{pucci1} satisfying \eqref{limsup} for $w$ as in {\upshape(iii)} . If either $u\geq0$ or $c^\alpha(x)\equiv0$, then $u$ is constant.
\item [(b)] Under the previous assumptions 
 on $H_s$, let $v\in\LSC(\R^d)$ be a viscosity supersolution to \eqref{pucci2} satisfying \eqref{liminf} for $W$ as in {\upshape(iii')}. If either $v\leq0$ or $c^\alpha(x)\equiv0$, then $v$ is constant.
\end{itemize}
\end{cor}
\begin{proof}
{ The proofs of (a) and (b) are consequences, respectively, of Theorems \ref{liosub} and \ref{liosuper}. The  operators $\mathcal{M}^-_{\lambda,\Lambda}((D^2_{\mathcal{X}}u)^*)$ and $\mathcal{M}^+_{\lambda,\Lambda}((D^2_{\mathcal{X}}u)^*)$ enjoy, respectively, the property \eqref{inf} and \eqref{sup} as a consequence of the property of duality (i.e., $\mathcal{M}^-_{\lambda,\Lambda}(M)=-\mathcal{M}^+_{\lambda,\Lambda}(-M)$
) and { the 
inequalities 
$$\mathcal{M}^-_{\lambda,\Lambda}(M+N)\leq \mathcal{M}^-_{\lambda,\Lambda}(M)+\mathcal{M}^+_{\lambda,\Lambda}(N), \quad \mathcal{M}^+_{\lambda,\Lambda}(M+N)\geq \mathcal{M}^-_{\lambda,\Lambda}(M)+\mathcal{M}^+_{\lambda,\Lambda}(N)$$
for every $M,N\in\Sym_d$, see \cite[Lemma 2.10
]{CC}. 
}
Moreover, they are positively $1$-homogeneous, so they satisfy the scalings (SC) and (SC').}

The comparison principle (ii) 
 holds for both equations in view of \cite[Example 4.6]{BG}. Finally, observe that when $c^\alpha\equiv0$, then $G(x,r,0,0)=0$ for every $x\in\Omega$, $r\in\R$, and $r\mapsto G(x,r,p,X)$ is constant for every $x,p,X$. 
\end{proof}

{ Corollary \ref{pucciminmax} concerns only operators that are either convex or concave with respect to the derivatives of the solution. However, we 
 will use it in the next Section \ref{fully} to study general fully nonlinear uniformly subelliptic equations.
 
 A different, although similar, class of extremal operators 
was  introduced by C. Pucci in the seminal paper \cite{Pucci66} (for the Euclidean Hessian). Here we consider them on the Hessian associated to a the vector fields $\mathcal{X}$. 
Consider}  for $\alpha>0$ the class of matrices
\begin{equation*}
\mathcal{B}_{\alpha}:=\{A\in\Sym_m:A\xi\cdot\xi\geq\alpha|\xi|^2,\mathrm{Tr}(A)=1,\forall\xi\in\R^m\} ,
\end{equation*}
and define
\begin{equation}\label{Pucci+}
\mathcal{P}^{+}_{\alpha}(M):=\sup_{A\in \mathcal{B}_{\alpha}}\mathrm{Tr}(-AM) , \qquad \mathcal{P}^{-}_{\alpha}(M)=\inf_{A\in\mathcal{B}_{\alpha}}\mathrm{Tr}(-AM)\ .
\end{equation}
As pointed out in \cite{Pucci66} (see also \cite[Chapter 17]{GT}), { these operators can be rewritten in terms of  the ordered eigenvalues of the matrix $M\in\Sym_m$ as follows }
\begin{equation}
\label{representation+}
\mathcal{P}^{+}_{\alpha}(M)=-\alpha\sum_{k=2}^{m
}e_k-[1-(m
-1)\alpha]e_1=-\alpha\mathrm{Tr}(M)-(1-m
\alpha)e_1 
\,,
\end{equation}
\begin{equation}
\label{representation-}
\mathcal{P}^{-}_{\alpha}(M)=-\alpha\sum_{k=1}^{m-1}e_k-[1-(m-1)\alpha]e_m=-\alpha\mathrm{Tr}(M)-(1-m\alpha)e_m\, .
\end{equation}
{ The Liouville properties of Corollary \ref{pucciminmax} hold under the same assumptions for the equations \eqref{pucci1} and \eqref{pucci2} with the operators $\mathcal{M}^-_{\lambda,\Lambda}((D^2_{\mathcal{X}}u)^*)$ and $\mathcal{M}^+_{\lambda,\Lambda}((D^2_{\mathcal{X}}u)^*)$ replaced, respectively, by $\mathcal{P}^{-}_{\alpha}((D^2_{\mathcal{X}}u)^*)$ and $\mathcal{P}^{+}_{\alpha}((D^2_{\mathcal{X}}u)^*)$. { In 
Section \ref{full_Hei}} { we will give some explicit results for extremal equations involving $\mathcal{P}^{\pm}$ on the Heisenberg group.}

}
%

\subsection{Fully nonlinear uniformly subelliptic equations}\label{fully}

In this section we consider the 
{ general fully nonlinear 
 subelliptic equation
\begin{equation}\label{2L}
G(x,u,D_{\mathcal{X}}u,(D_{\mathcal{X}}^2u)^*)=0 \quad \text{ in }\R^d
\end{equation}
when { $G : \R^d\times\R\times\R^m\times\mathcal S_m\to\R$} satisfies 
the following form of uniform ellipticity}
\begin{equation}\label{US}
\mathcal{M}^-_{\lambda,\Lambda}(M-N)\leq G(x,r,p,M)-G(x,r,p,N)\leq \mathcal{M}^+_{\lambda,\Lambda}(M-N)
\end{equation}
for every $(x,r,p)\in \Omega\times\R\times\R^m$ and $M,N\in\Sym_m$ with $N\geq0$. By taking $N=0$ we get
\begin{equation*}
\mathcal{M}^-_{\lambda,\Lambda}(M)\leq G(x,r,p,M)-G(x,r,p,0)\leq \mathcal{M}^+_{\lambda,\Lambda}(M)\ ,
\end{equation*}
and, as a consequence, by setting $H(x,r,p):=G(x,r,p,0)$, one can infer Liouville results for viscosity subsolutions and supersolutions to { \eqref{2L} by comparison with the equation \eqref{pucci1} and \eqref{pucci2} of the previous section.
For this purpose we assume either}
\begin{equation}\label{G>H}
G(x,r,p,0)\geq H_i(x,r,p) \quad \forall \, x \in \R^d , r\in\R , p\in\R^m ,
\end{equation}
for a concave Hamiltonian of the form \eqref{Hi}, { or 
\begin{equation}\label{G<H}
G(x,r,p,0)\leq H_s(x,r,p)
\end{equation}
for a convex $H_s$ as in \eqref{Hs}.
}
\begin{cor}
\label{fullymin}
 { Assume the vector fields $\mathcal{X}$ 
are $C^{1,1}$ 
and satisfy the H\"ormander condition (H), \eqref{blip}, \eqref{cass}, and the ellipticity condition \eqref{US}.}
\begin{itemize}
\item [(a)]
Let \eqref{G>H} hold and 
 $u\in\USC(\R^d)$ be a viscosity subsolution to \eqref{2L} satisfying \eqref{limsup} for $w$ as in {\upshape(iii)}. If either $u\geq0$ or $c^\alpha(x)\equiv0$, then $u$ is constant.
\item [(b)]
Let \eqref{G<H} hold and $v\in\LSC(\R^d)$ be a viscosity supersolution to \eqref{2L} satisfying \eqref{liminf} for $W$ as in {\upshape(iii')}. If either $v\leq0$ or $c^\alpha(x)\equiv0$, then $v$ is constant.
\end{itemize}
\end{cor}
\begin{proof}
It is sufficient to observe that $u$ and $v$ satisfies the differential inequalities
\[
\mathcal{M}^-_{\lambda,\Lambda}((D^2_{\mathcal{X}}u)^*)+H_i(x,u,D_{\mathcal{X}}u)\leq0 \text{ in }\R^d \,,
\]
{ \[
\mathcal{M}^+_{\lambda,\Lambda}((D^2_{\mathcal{X}}v)^*)+H_s(x,v,D_{\mathcal{X}}v)\geq0 \text{ in }\R^d \,,
\] }
and apply Corollary \ref{pucciminmax}.
\end{proof}

{ The same kind of result holds for equations of the form
{
\begin{equation}\label{non-h-grad}
G(x,u,Du, (D_{\mathcal{X}}^2u)^*) =0\text{ in }\R^d\ ,
\end{equation}
for $G : \R^d\times\R\times\R^d\times\mathcal S_m\to\R$,} where the dependence is on the Euclidean gradient $Du$ instead of the horizontal one $D_{\mathcal{X}}u$. We assume $G$ satisfies \eqref{US} and either 
\begin{equation}\label{G>Hp}
{ G(x,r,p,0)}\geq H_i(x,r,p) ,  \quad \forall \, x, p\in \R^d , r\in\R ,
\end{equation}
for a concave Hamiltonian of the form \eqref{Hi}, but with  $b^\alpha : \R^d \to \R^d$ a vector field in $\R^d$,  or 
\begin{equation}\label{G<Hp}
{ G(x,r,p,0)}\leq H_s(x,r,p)     ,  \quad \forall \, x, p\in \R^d , r\in\R ,
\end{equation}
for a convex Hamiltonian of the form \eqref{Hs} with  $b^\alpha : \R^d \to \R^d$. On $b^\alpha$ and $c^\alpha$ we make the same assumptions \eqref{blip}, \eqref{cass}, and the fields $\mathcal{X}$ 
are $C^{1,1}$ 
and satisfy 
(H). The arguments leading to Corollaries \ref{pucciminmax} and \ref{fullymin} give the following.
\begin{cor} 
\label{eu-grad}
Under the conditions listed above a subsolution (resp., supersolution) of \eqref{non-h-grad} with assumption \eqref{G>Hp} (resp.,  \eqref{G<Hp}) verifies the Liouville property { \upshape (a)} (resp.,  { \upshape (b)})of Corollary \ref{fullymin}.
\end{cor}
} 
\subsection{{Normalized 
$p$-Laplacian.}}
The result of the last section encompasses degenerate equations of the form
\[
{-}|D_\X u|^{2-p}\mathrm{div}_\X(|D_\X u|^{p-2}D_\X u)=0 \quad \text{ in }\R^d ,
\]
{ where $\mathrm{div}_\X$ is the intrinsic divergence over the fields of the family $\X$.}
In fact the operator $E(D_\X u, (D^2_\X u)^*):=-|D_\X u|^{2-p}\mathrm{div}_\X(|D_\X u|^{p-2}D_\X u)$, called normalized or game-theoretic
$p$-Laplacian, can be 
 rewritten as $-\mathrm{Tr}[A(D_\X u)(D^2_\X u)^*]$, 
with 
\[
A(D_\X u)=I_m+(p-2)\frac{D_\X u\otimes D_\X u}{|D_\X u|^2} \,.
\]
In other words, 
$$E(D_\X u, (D^2_\X u)^*)=-\Delta_\X u-(p-2)|D_\X u|^{-2}\Delta_{\X,\infty}u ,$$
 where 
$\Delta_{\X,\infty}$ is the  $\infty$-Laplacian operator over the fields $\X$. 
 It is immediate to see that {
\[
\min\{1,p-1\}|\xi|^2\leq A(D_\X u)\xi\cdot \xi\leq\max\{1,p-1\}|\xi|^2\ ,
\]
showing that $E$ is uniformly subelliptic for $p\in(1,\infty)$.
Therefore such nonlinear operators 
can be compared with Pucci's extremal operators $\mathcal{M}^{\pm}_{\lambda,\Lambda}$ with $\lambda=\min\{1,p-1\}$ and $\Lambda=\max\{1,p-1\}$ over H\"ormander vector fields. They were studied recently by several authors, see e.g. 
  \cite{Pina} for the case of Carnot groups, and the references therein. The game theoretic $p$-Laplace operator on $\X$ is the sublaplacian if $p=2$, whereas for  $p=1$ it drives the evolutive equation 
 describing the motion of level sets by sub-Riemannian mean curvature.}

\section{
{ The Heisenberg vector fields}} 
\label{heisenberg}
The aim of this section is to specialize the results obtained in the previous one to viscosity subsolutions of \eqref{2L} fulfilling \eqref{US} over Heisenberg vector fields. We briefly recall some standard facts on the Heisenberg group. For further details we refer the reader to the monograph \cite{BLU}. The Heisenberg group $\He^d$ can be identified with $(\R^{2d+1},\circ)$, where $2d+1$ stands for the topological dimension and the group law $\circ$ is defined by
\begin{equation*}
x\circ y=\left(x_1+y_1,...,x_{2d}+y_{2d},x_{2d+1}+y_{2d+1}+2\sum_{i=1}^d(x_iy_{i+d}-x_{i+d}y_i ) \right) .
\end{equation*}
{
We denote with $x$  a point of $\R^{2d+1}$ and set
\[
x_H:=(x_1,...,x_{2d})\ .
\]
}
The $d$-dimensional Heisenberg algebra is the Lie algebra spanned by the $m=2d$ vector fields 
\begin{equation*}
X_i=\partial_i+2x_{i+d}\partial_{2d+1}\ ,
\end{equation*}
\begin{equation*}
X_{i+d}=\partial_{i+d}-2x_{i}\partial_{2d+1}\ ,
\end{equation*}
for $i=1,...,d$. 
Such vector fields satisfy the commutation relations
\[
[X_i,X_{i+d}]=-4\partial_{2d+1}\text{ and }[X_i,X_j]=0\text{ for all $j\neq i+d$, }i\in\{1,...,d\}\ .
\]
and are 1-homogeneous with respect to the family of (anisotropic) dilations
\[
\delta_\lambda(x)=(\lambda x_1,...,\lambda x_{2d},\lambda^2x_{2d+1}) \ ,\lambda>0\ ,
\]

Following \cite[Definition 5.1.1]{BLU}, it is useful to consider the following homogeneous norm defined via the stratification property of $\He^d$
\begin{equation}\label{homoH}
\rho(x)=\left(\left(\sum_{i=1}^{2d}(x_i)^2\right)^2+x_{2d+1}^2\right)^{\frac14}\ .
\end{equation}
which is 1-homogeneous with respect to the previous group of dilations. { We emphasize that 
 this norm is easier to compute than 
  the Carnot-Carath\'eodory norm.} 

\subsection{Fully nonlinear PDEs on the Heisenberg group}
\label{full_Hei}
In the next result we provide sufficient conditions for the validity of the Liouville property for viscosity subsolutions to \eqref{pucci1}. We search the Lyapunov functions of property (iii) and (iii') among the radial ones (e.g. $w=\log\rho$ and $W=-\log\rho$). 
Here and in the next examples we exploit a classical chain rule to compute the horizontal gradient and Hessian of a ``radial" function with respect to the homogeneous norm $\rho$. Indeed, for a sufficiently smooth radial function $f=f(\rho)$ and given a system of vector fields $\X=\{X_1,...,X_m\}$, we have
\begin{equation*}
D_{\X}f(\rho)=f'(\rho)D_{\X}\rho
\end{equation*}
and
\begin{equation*}
D^2_{\X}f(\rho)=f'(\rho)D^2_{\X}\rho+f^{''}(\rho)D_{\X}\rho\otimes D_{\X}\rho\ .
\end{equation*}
In this section we denote the Heisenberg horizontal gradient and symmetrized Hessian by $D_{\mathbb{H}^d}$ and $(D^2_{\mathbb{H}^d})^*$.
We premise the following auxiliary result taken from \cite[Lemma 3.1 and Lemma 3.2]{CTchou}.
\begin{lemma}\label{hesH}
Let $\rho$ be defined in \eqref{homoH}. Then, { for $|x_H|\neq0$,
\[
D_{\mathbb{H}^d}\rho = \frac{\eta}{\rho^3} , \qquad |D_{\mathbb{H}^d}\rho|^2=\frac{|x_H|^2}{\rho^2}\leq 1\, ,
\]
where $\eta\in\R^{2d}$ 
is defined by 
\begin{equation}
\label{eta}
\eta_i:=x_i|x_H|^2+x_{i+d}x_{2d+1}\ , \qquad
\eta_{i+d}:=x_{i+d}|x_H|^2-x_{i}x_{2d+1} .
\end{equation}
for $i=1,...,d$. Moreover }
\[
D^2_{\mathbb{H}^d}\rho=-\frac{3}{\rho}D_{\mathbb{H}^d}\rho\otimes D_{\mathbb{H}^d}\rho+\frac{1}{\rho}|D_{\mathbb{H}^d}\rho|^2I_{2d}+\frac{2}{\rho^3}\begin{pmatrix}B & C\\-C & B\end{pmatrix}\ ,
\]
where the matrices $B=(b_{ij})$ and $C=(c_{ij})$ are defined as follows
\[
b_{ij}:=x_ix_j+x_{d+i}x_{d+j}\ ,\qquad c_{ij}:=x_ix_{d+j}-x_jx_{d+i}
\]
for $i,j=1,...,d$ (in particular $B=B^T$ and $C=-C^T$).
 In addition, for a radial function $f=f(\rho)$ we have
\[
D^2_{\mathbb{H}^d}f(\rho)=\frac{f'(\rho)|D_{\mathbb{H}^d}\rho|^2}{\rho}I_{2d}+2\frac{f'(\rho)}{\rho^3}\begin{pmatrix}B & C\\-C & B\end{pmatrix}+\left(f''(\rho)-3\frac{f'(\rho)}{\rho}\right)D_{\mathbb{H}^d}\rho\otimes D_{\mathbb{H}^d}\rho\ .
\]
{ and its eigenvalues are $f''(\rho)|D_{\mathbb{H}^d}\rho|^2, 3f'(\rho)\frac{|D_{\mathbb{H}^d}\rho|^2}{\rho}$, which are simple, and $f'(\rho)\frac{|D_{\mathbb{H}^d}\rho|^2}{\rho}$ with multiplicity $2d-2$}.
\end{lemma}
\begin{thm}\label{Cor1H}
Let $\mathcal{X}=\{X_1,....,{ X_{2d}}
\}$ be the system of vector fields generating the Heisenberg group $\He^d$. Assume that \eqref{blip} and \eqref{cass} are in force and
\begin{equation}\label{condcor1}
\sup_{\alpha\in A}\{b^\alpha(x)\cdot \frac{\eta}{|x_H|^2}-c^\alpha(x)\frac{\rho^4}{|x_H|^2}\log\rho\}\leq \lambda-\Lambda(Q-1)
\end{equation}
for $\rho$ sufficiently large and $|x_H|\neq0$, where $Q=2d+2$ is the homogeneous dimension of $\mathbb{H}^d$, $b^\alpha(x)$ takes values in $\R^{2d}$, and $\eta=(\eta_i,\eta_{i+d})$ is defined { by \eqref{eta}}. 
%
\begin{itemize}
\item[(A)] Let $u\in\USC(\R^{2d+1})$ be a viscosity subsolution of \eqref{pucci1} such that
\begin{equation}
\label{grow+}
\limsup_{|x|\rightarrow\infty}\frac{u(x)}{\log\rho(x)}\leq0\ .
\end{equation}
If either $c^\alpha(x)\equiv0$ or $u\geq0$, then $u$ is a constant.
\item[(B)] Let $v\in\LSC(\R^{2d+1})$ be a viscosity supersolution of \eqref{pucci2} such that
\begin{equation}
\label{grow-}
\liminf_{|x|\rightarrow\infty}\frac{v(x)}{\log\rho(x)}\geq0\ .
\end{equation}
If either $c^\alpha(x)\equiv0$ or $v\leq0$, then $v$ is a constant.
\end{itemize}
\end{thm}
\begin{rem}
When $b\equiv c\equiv 0$ and $\lambda=\Lambda=1$ (i.e. \eqref{pucci1} becomes $-\Delta_{\He^d}u=0$) condition \eqref{condcor1} gives $\lambda-\Lambda(Q-1)=(2-Q)\geq0$, which is not satisfied in  the Heisenberg group because $Q\geq4$. This is consistent with the failure of the Liouville property for sub- and supersolutions of the Heisenberg sub-Laplacian {  that we prove in Section \ref{cex} below.}
\end{rem}
\begin{proof}
We only have to check property {\upshape(iii)} { for 
 the Lyapunov function 
 $w(x)
 =\log\rho(x)$. Note} that $\lim_{|x|\to\infty}w
 (x)=\infty$ because 
 $\rho\to\infty$ as $|x|\to\infty$.
{ By Lemma \ref{hesH}  applied to the radial function $w$} 
the eigenvalues of $(D^2_{\He^d}w)^*$ are
\[
-\frac{|x_H|^2}{\rho^4} \,\text{ and }\, 3\frac{|x_H|^2}{\rho^4},\quad \text{which are simple} ,
\]
and
\[
\frac{|x_H|^2}{\rho^4}\text{ with multiplicity $2d-2$ }
\]
when $|D_{\He^d}\rho|\neq0$. 
 { Otherwise all the eigenvalues vanish identically and  $w
$ is trivially a supersolution to \eqref{pucci1} because $\mathcal{M}^-_{\lambda,\Lambda}((D^2_{\He^d}w)^*)=0$ 
 and $c^\alpha u\geq0$.}
 Hence, we are able to compute the Pucci's minimal operator at points where $|x_H|\neq0$ owing to Lemma \ref{hesH} as
\begin{equation*}
\mathcal{M}^-_{\lambda,\Lambda}((D^2_{\He^d}w)^*)=\{-\Lambda(2d+1)+\lambda\}\frac{|x_H|^2}{\rho^4}\ .
\end{equation*}
Thus, $w$ is a supersolution at all points where
\begin{equation*}
\{-\Lambda(2d+1)+\lambda\}\frac{|x_H|^2}{\rho^4}+\inf_{\alpha\in A}\{c^\alpha(x)\log\rho-b^\alpha(x)\cdot\frac{\eta}{\rho^4}\}\geq0\ ,
\end{equation*}
{ because $D_{\He^d}w= \eta/\rho^4$ by Lemma \ref{hesH}. }  
In particular, this inequality holds when $\rho$ is sufficiently large under condition \eqref{condcor1} by recalling that $Q=2d+2$. Similarly, one can check that \eqref{condcor1} implies that the function $W(\rho)=-\log\rho$ is a subsolution to \eqref{pucci2} for $|x|$ sufficiently large at points where $|D_{\He^d}\rho|\neq0$. Therefore Corollary \ref{pucciminmax} gives the conclusion.
\end{proof}
\begin{rem}
Condition \eqref{condcor1} is comparable to that obtained in \cite[condition (2.17)]{BC}, but here typical quantities of Carnot groups appear. One may think that the ratio
\[
\frac{\rho^4}{|x_H|^2}=\frac{|x_H|^4+|x_V|^2}{|x_H|^2}
\]
plays exactly the same role as $|x|^2$ in \cite[condition (2.17)]{BC}, while the dimension $d$ of the Euclidean setting is precisely replaced by its sub-Riemannian counterpart $Q$, as expected.
\end{rem}
\begin{rem}
A simple condition that implies \eqref{condcor1}, and therefore the Liouville property,  is 
\begin{equation*}
\limsup_{|x|\rightarrow\infty}\sup_{\alpha\in A} b^\alpha(x)\cdot \frac{\eta}{|x_H|^2} <\lambda-\Lambda(Q-1)\ ,
\end{equation*}
since $c\geq0$. Compare the above condition to that in \cite[Remark 2.4]{BC}: $Q$ replaces the dimension $d$ of the Euclidean case and $x\in\R^d$ is replaced by the vector $\eta/|x_H|^2\in\R^{2d}$, where $\eta$ is defined by \eqref{eta}.
\end{rem}
{ We can now give explicit conditions for  the Liouville properties 
for the general subelliptic equation \eqref{2L} on the Heisenberg group.}
\begin{cor}
Assume that the operator $G$ { satisfies \eqref{US}, 
$\X=\{X_1,...,X_{2d}\}$}  are the Heisenberg vector fields,  
and \eqref{blip}, \eqref{cass}, and \eqref{condcor1} are satisfied. 

\noindent {\upshape(A)} Assume 
 \eqref{G>H} and $u\in\USC(\R^{2d+1})$ is a subsolution of \eqref{2L} satisfying  { \eqref{grow+}.}
 If either $c^\alpha(x)\equiv0$ or $u\geq0$, then $u$ is constant.

\noindent {\upshape(B)} Assume 
 \eqref{G<H} and $v\in\LSC(\R^{2d+1})$ is a supersolution of \eqref{2L} satisfying  { \eqref{grow-}.}
 If either $c^\alpha(x)\equiv0$ or $v\leq0$, then $v$ is constant.
\end{cor}
\begin{proof}
It is enough to exploit that $u$ { (resp., $v$)} is a subsolution to \eqref{pucci1} (resp., a supersolution to \eqref{pucci2}) over the Heisenberg group and then apply Theorem \ref{Cor1H}-(A) (resp., Theorem \ref{Cor1H}-(B)).
\end{proof}

We specialize the last corollaries to a class of examples in order to compare with those in \cite{BC}. Consider again 
the fully nonlinear uniformly subelliptic PDE \eqref{2L} 
and assume that either
\begin{equation}
\label{hypH}
G(x,r,p,0)\geq-\bar{b}(x)\cdot p-g(x)|p|+\bar{c}(x)r\,,
\end{equation}
{ or 
\begin{equation}
\label{hypH2}
{G}(x,r,p,0)\leq-\bar{b}(x)\cdot p+g(x)|p|+\bar{c}(x)r \,,
\end{equation}
where $\bar{b}:\R^{2d+1}\rightarrow\R^{2d}$ and $g:\R^{2d+1}\rightarrow\R$ are locally Lipschitz, $\bar{c}$ is continuous,} $g\geq0$, and $\bar{c}\geq0$. 
\begin{cor}
\label{corollbcg}
Assume that the operator $G$ in \eqref{2L} satisfies \eqref{US} and 
\begin{equation}\label{condcor1bis}
\bar{b}(x)\cdot \frac{\eta}{|x_H|^2}+g(x)\frac{|\eta|}{|x_H|^2}\leq\bar{c}(x)\frac{\rho^4}{|x_H|^2}\log\rho+\lambda-\Lambda(Q-1) \,,
\end{equation}
{ for $\rho$ sufficiently large and $|x_H|\neq0$, where $\eta$ is defined by \eqref{eta} and $Q=2d+2$.}

\noindent {\upshape(A)} Suppose  that \eqref{hypH} holds and
$u\in\USC(\R^{2d+1})$ is a viscosity subsolution of \eqref{2L} satisfying \eqref{grow+}. If
either $c^\alpha(x)\equiv0$ or $u\geq0$, then $u$ is a constant. 

{ \noindent {\upshape(B)} Suppose  that \eqref{hypH2} holds and
 $v\in\LSC(\R^{2d+1})$ is a viscosity supersolution of \eqref{2L} satisfying \eqref{grow-}}. If
either $v\leq0$ or $c^\alpha(x)\equiv0$, then $v$ is constant.
\end{cor}
\begin{proof} { As in the proof of Theorem \ref{Cor1H} we must only check that $w=\log\rho$ is a supersolution.
Observe that $-|D_{\He^d}w|=-|\sigma^TDw|=\min_{|\alpha|=1}\{-\alpha\cdot \sigma^TDw\}$. Hence we can write the right-hand side of the inequality \eqref{hypH} with $p=D_{\He^d}w$ as
\begin{equation*}
\inf_{\alpha\in A}\{\bar{c}w-(\bar{b}+g\alpha)\cdot\sigma^TDw\}\ ,
\end{equation*}
}
where $A=\{\alpha\in\R^{2d}:|\alpha|=1\}$. Moreover 
{ $D_{\He^d}w=\frac{1}{\rho^4}\eta$ by Lemma \ref{hesH}. 
Then $w$ is a supersolution where
\begin{equation*}
\{-\Lambda(2d+1)+\lambda\}\frac{|x_H|^2}{\rho^4}+ \bar{c}\log\rho+\inf_{\alpha\in A}\{-(\bar{b}+g\alpha)\cdot\frac{\eta}{\rho^4}\}\geq0 \,,
\end{equation*}
and this inequality is satisfied for $\rho$  large enough if \eqref{condcor1bis} holds.

Arguing in a similar manner we prove (B) by showing that $W=-\log\rho$ is a subsolution.} 
\end{proof}
{ \begin{ex} [Schr\"odinger-type equations] \label{schro} For nonnegative subsolutions of the equation
\[
\mathcal{M}^-_{\lambda,\Lambda}((D^2_{\He^d}u)^*) + \bar{c}(x)u =0 \quad \text{ in }\R^{2d+1}
\]
the Liouville property holds if 
\[
\liminf_{
|x|\to\infty} \bar{c}(x)\frac{\rho^4(x)}{|x_H|^2}\log\rho(x)>{\Lambda(Q-1)-\lambda}\ ,
\]
a results that appears to be new even in the linear case $\lambda=\Lambda$.
\end{ex} }
{ \begin{ex} [A horizontal Ornstein-Uhlenbeck equation] Consider  subsolutions of the equation 
\[
\mathcal{M}^-_{\lambda,\Lambda}((D^2_{\He^d}u)^*) - \gamma(x) \eta(x) \cdot D_{\He^d}u =0 \quad \text{ in }\R^{2d+1}\ 
\]
where $\gamma(x) >0$ and $\eta$ is defined by \eqref{eta}, i.e.,
\[
\eta(x)=x_H|x_H|^2+ x_{2d+1}x_H^\perp \,,\qquad x_H^\perp:=(x_{d+1},...,x_{2d}, -x_1,...,-x_d) \,.
\]
Since $\eta\cdot\eta
=|x_H|^2 \rho^4$, the condition \eqref{condcor1bis} becomes
\[
\liminf_{|x|\to\infty} \gamma(x)\rho^4(x)>{\Lambda(Q-1)-\lambda}
\]
and then the Liouville property holds. 
\end{ex} }

We end this subsection with a result on the following equations driven by the Pucci's extremal operators $\mathcal{P}^\pm_\lambda$ defined by \eqref{Pucci+} (here $\lambda=\alpha$)
\begin{equation}\label{pucci1bis}
\mathcal{P}^-_{\lambda}((D^2_{\mathcal{X}}u)^*)+H_i(x,u,D_{\mathcal{X}}u)=0\quad \text{ in }\R^{2d+1}\ ,
\end{equation}
\begin{equation}\label{pucci2bis}
\mathcal{P}^+_{\lambda}((D^2_{\mathcal{X}}u)^*)+H_s(x,u,D_{\mathcal{X}}u)=0\quad \text{ in }\R^{2d+1}\ .
\end{equation}
Sufficient conditions for the Liouville property can be 
 obtained by comparing 
  $\mathcal{P}^\pm$ with $\mathcal{M}^\pm$ as follows 
\begin{equation*}
\mathcal{P}^+_{\lambda}(M)\leq\mathcal{M}^{+}_{\lambda,\lambda+(1-d\lambda)}(M) \,,\quad \mathcal{P}^-_{\lambda}(M)\geq\mathcal{M}^{-}_{\lambda,\lambda+(1-d\lambda)}(M)\,, \quad \forall\, M\in\Sym_{2d} \,.
\end{equation*}
However, 
by exploiting the representation formulas \eqref{representation+} and \eqref{representation-} 
for $\mathcal{P}^\pm$ we can get 
 optimal sufficient conditions. 
\begin{cor}\label{Cor1P}
Let $\mathcal{X}=\{X_1,....,X_{2d}
\}$ be the system of vector fields generating the Heisenberg group $\He^d$. Assume 
 \eqref{blip}, \eqref{cass}, 
  and
\begin{equation}\label{condcor1p}
\sup_{\alpha\in A}\{b^\alpha(x)\cdot \frac{\eta}{|x_H|^2}-c^\alpha(x)\frac{\rho^4}{|x_H|^2}\log\rho\}\leq 4d\lambda -3
\end{equation}
for $\rho$ sufficiently large and 
$|D_{\He^d}\rho|\neq0$.
{ Then the same conclusions as in Theorem \ref{Cor1H}  hold for subsolutions of \eqref{pucci1bis} 
 and supersolutions of \eqref{pucci2bis}. }
\end{cor}
\begin{proof}
The proof is 
 the same as Theorem \ref{Cor1H} using the Lyapunov function $w(\rho)=\log\rho$ and the 
  formulas 
   \eqref{representation+} and \eqref{representation-}.
By the expression of the eigenvalues of $(D^2_{\He^d}w)^*$ in the proof of Theorem \ref{Cor1H} one finds
\[
\mathcal{P}^-_\lambda((D^2_{\He^d}w)^*)=(4d\lambda-3)\frac{|x_H|^2}{\rho^4}\ .
\]
Similarly, one uses $W=-\log\rho$ as Lyapunov function for the maximal operator $\mathcal{P}^+_\lambda$.
\end{proof}
\begin{rem}
Condition \eqref{condcor1p} is better than \eqref{condcor1} with $\Lambda=\lambda+(1-2d\lambda)$ and $\lambda<\frac{1}{2d}$, since
\begin{equation*}
-2d\lambda-(1-2d\lambda)(2d+1)<-3+4d\lambda\ .
\end{equation*}
\end{rem}

\subsection{Comparison with the literature and sharpness of the conditions}\label{cex}
In this section we make a comparison with the results in the literature, showing the sharpness of our conditions and those of \cite{BC, CLeoni, CTchou} via several counterexamples.
\subsubsection{{ The Euclidean case}}
{ Corollary 2.4 of \cite{BC} 
 states a Liouville-type result 
 that in the case without lower order terms holds } { for the inequalities
\[
\mathcal{M}^{-}_{\lambda,\Lambda}(D^2u)\leq 0\text{ in }\R^d , \qquad \mathcal{M}^{+}_{\lambda,\Lambda}(D^2u)\geq 0\text{ in }\R^d ,
\]
the former for viscosity subsolutions bounded above, and the second for supersolutions bounded from below, 
  when $d\leq \frac{\lambda}{\Lambda}+1$. 
  This complements the result  of \cite{CLeoni} on \eqref{M+} and \eqref{M-} recalled in the Introduction,
  but with a more restrictive condidion on $d$, which is, however, still sharp for the Lapalcian (${\lambda}={\Lambda}$).
  The next  counterexample shows that the 
 inequalities can have nonconstant solutions when $d>\frac{\lambda}{\Lambda}+1$.}
%
\begin{cex}
For $d\geq2$ { set 
$\beta:=\frac{\Lambda}{\lambda}(d-1)+1 
$ 
and consider }the function
\[
u_2(x)=\begin{cases}
\frac18[\beta(\beta-2)|x|^4-2(\beta^2-4)|x|^2+\beta(\beta+2)]&\text{ if }|x|<1\ ,\\
\frac{1}{|x|^{\beta-2}}&\text{ if }|x|\geq1\ .
\end{cases}
\]
{ Since it 
is radial, 
 the eigenvalues of the Hessian matrix can be 
 computed 
 by \cite[Lemma 3.1]{CLeoni} and one checks that it
is } a 
 classical solution to $\mathcal{M}^+_{\lambda,\Lambda}(D^2u_2)\geq0$ in $\R^d$. 
 { Moreover it is bounded and not constant if $\beta>2$, which is equivalent to $d>\lambda/\Lambda+1$, so the Liouville property for supersolutions to $\mathcal{M}^+(D^2u)=0$ is false in this case}. Similarly, $v_2=-u_2$ gives a counterexample for solutions to $\mathcal{M}^-_{\lambda,\Lambda}(D^2v_2)\leq0$ in $\R^d$. 

\end{cex}
{ \begin{rem} 
The paper \cite{CLeoni} studies a similar but different problem with respect to \cite{BC}, namely, 
the Liouville property 
for viscosity supersolutions to $\mathcal{M}^-_{\lambda,\Lambda}(D^2u)=0$ in $\R^d$ and for 
subsolutions to $\mathcal{M}^+_{\lambda,\Lambda}(D^2u)=0$. 
They prove it under the  less restrictive condition $d\leq \frac{\Lambda}{\lambda}+1$ \cite[Theorem 3.2]{CLeoni}. Note, however, that their theorem cannot be applied to general uniformly elliptic operators via the inequalities \eqref{US}, whereas the results in \cite{BC} allow such application.
 The next example shows that also the condition in \cite{CLeoni} is optimal.
 \end {rem} 
}
 \begin{cex} [From 
 \cite{CLeoni}] { Set $\alpha:=\frac{\lambda}{\Lambda}(d-1)+1$ and consider  the function}
 \[
u_3(x)=\begin{cases}
-\frac18[\alpha(\alpha-2)|x|^4-2(\alpha^2-4)|x|^2+\alpha(\alpha+2)]&\text{ if }|x|<1\ ,\\
-\frac{1}{|x|^{\alpha-2}}&\text{ if }|x|\geq1\ ,
\end{cases}
\]
which is a 
 classical solution to $\mathcal{M}^+_{\lambda,\Lambda}(D^2u_3)\leq0$ in $\R^{d}$. { Moreover it is bounded and not constant if $\alpha>2$, which is equivalent to $d>\Lambda/\lambda+1$}. Similarly, $v_3=-u_3$ yields a counterexample for the corresponding property for the minimal operator.
 \end{cex}

\subsubsection{{ The Heisenberg case: sublaplacians.}}

Liouville's theorem for classical harmonic functions on the Heisenberg group is a consequence of the Harnack inequality, see, e.g., \cite[Theorem 8.5.1]{BLU}, { or mean-value formulas \cite[Theorem 3.1]{ICDCsub}}. However, the Liouville property for classical subsolutions (resp., supersolutions) bounded from above (resp., below) of
\[
-\Delta_{\He^d}u=0\quad \text{ in }{ \mathbb{H}^d\simeq} \R^{2d+1}
\]
is false { for all dimensions $d$, as the next example shows. 
} 
 We recall that  $Q:=2d+2$ is the homogeneous dimension of $\mathbb{H}^d$ and $\rho(x)$ is the homogeneous norm defined in \eqref{homoH}. 
\begin{cex}
The function
\[
\tilde u(x)=\begin{cases}
\frac18[Q(Q-2)\rho^4-2(Q^2-4)\rho^2+Q(Q+2)]&\text{ if }\rho\leq1\ ,\\
\frac{1}{\rho^{Q-2}}&\text{ if }\rho\geq1\ ,
\end{cases}
\]
is a bounded classical supersolution to $-\Delta_{\He^d}u=0$ in $\R^{2d+1}$. Indeed, when $\rho\leq 1$ one  applies Lemma \ref{hesH} to  the radial function
\[
\tilde u=\tilde f(\rho)=\frac18[Q(Q-2)\rho^4-2(Q^2-4)\rho^2+Q(Q+2)],
\]
and gets, 
at points where $|D_{\mathbb{H}^d}\rho|\neq0$, 
\begin{multline*}
-\Delta_{\mathbb{H}^d} \tilde u=-\mathrm{Tr}(D^2_{\mathbb{H}^d} \tilde f(\rho))\\
=-\frac{Q-2}{2\rho^2}|x_H|^2\left\{[3Q\rho^2-(Q+2)]+3[Q\rho^2-(Q+2)]+(2d-2)[Q\rho^2-(Q+2)]\right\}\\
=-\frac{Q-2}{2\rho^2}|x_H|^2Q(\rho^2-1)(Q+2)\geq 0 \,,
\end{multline*}
due to the fact that $\rho^2\leq1$ and $Q\geq4$. At points where $|x_H|=0$ all the eigenvalues of $D^2_{\mathbb{H}^d} \tilde f(\rho)$ vanish and hence $\tilde u$ is a solution of the sub-Laplace equation. When $\rho\geq1$, instead, one observes that $\tilde u$ is  the fundamental solution of the sub-Laplace equations on the Heisenberg group found by G.B. Folland \cite{Folland}.
Similarly, $v=-\tilde u$ gives a bounded subsolution to $-\Delta_{\He^d}u=0$ in $\R^{2d+1}$. 
\end{cex}

\subsubsection{{ The Heisenberg case: fully nonlinear operators.}}
{ The Liouville property in this context was first studied by Cutr\`i and Tchou  \cite
{CTchou}}
for viscosity supersolutions bounded from below of  $\mathcal{M}^-_{\lambda,\Lambda}((D^2_{\He^d}u)^*)$ $=0$ in $\R^{2d+1}$ and for 
subsolutions bounded from above to $\mathcal{M}^+_{\lambda,\Lambda}((D^2_{\He^d}u)^*)=0$. 
{ Their  Theorem 5.2 states that such functions are constant provided that $Q\leq \frac{\Lambda}{\lambda}+1$. The next is a new example showing that this condition is sharp.}



 \begin{cex}\label{cexalphatilde}
Set $\tilde \alpha:=\frac{\lambda}{\Lambda}(Q-1)+1$. 
We show that for $\tilde\alpha>2$, { i.e., $Q> \frac{\Lambda}{\lambda}+1$,}
 \[
u_4(x)=\begin{cases}
-\frac18[\tilde \alpha(\tilde \alpha-2)\rho^4-2(\tilde \alpha^2-4)\rho^2+\tilde \alpha(\tilde \alpha+2)]&\text{ if }\rho<1\ ,\\
-\frac{1}{\rho^{\tilde \alpha-2}}&\text{ if }\rho\geq1\ ,
\end{cases}
\]
is a bounded from above classical solution to $\mathcal{M}^+_{\lambda,\Lambda}((D^2_{\He^d}u_4)^*)\leq0$ in $\R^{2d+1}$ and it is not constant. Indeed, denote by $u_4(x)=f_4(\rho)$. For $\rho<1$ we have 
\[
f'_4(\rho)=-\frac{\tilde \alpha-2}{2}\rho[\tilde \alpha\rho^2-(\tilde \alpha+2)]\ ,
\]
and
\[
f''_4(\rho)=-\frac{\tilde \alpha-2}{2}[3\rho^2\tilde\alpha-(\tilde\alpha+2)]
\]
Recalling that $|D_{\He^d}\rho|^2=|x_H|^2/\rho^2$, by Lemma \ref{hesH} the eigenvalues of the radial function $f_4(\rho)$
\[
e_1=|D_{\He^d}\rho|^2f''_4(\rho)=-\frac{\tilde \alpha-2}{2\rho^2}|x_H|^2[3\rho^2\tilde\alpha-(\tilde \alpha+2)] ,
\]
\[
e_2=3|D_{\He^d}\rho|^2\frac{f'_4(\rho)}{\rho}=-3\frac{\tilde \alpha-2}{2\rho^2}|x_H|^2[\tilde \alpha\rho^2-(\tilde \alpha+2)] ,
\]
which are both simple, and
\[
e_3=|D_{\He^d}\rho|^2\frac{f'_4(\rho)}{\rho}=-\frac{\tilde \alpha-2}{2\rho^2}|x_H|^2[\tilde \alpha\rho^2-(\tilde \alpha+2)] ,
\]
which has multiplicity $2d-2$. 
Observe that, when $\rho<1$ and $\tilde \alpha>2$, the eigenvalues $e_2,e_3$ are always positive. Moreover, 
for $\rho^2\leq\frac{\tilde \alpha+2}{3\tilde\alpha}<1$, also $e_1$ is positive and hence $\mathcal{M}^+_{\lambda,\Lambda}((D^2_{\He^d}u_4)^*)\leq0$. When $1>\rho^2>\frac{\tilde \alpha+2}{3\tilde\alpha}$, $e_1<0$, and hence by Lemma \ref{hesH}
\begin{multline*}
\mathcal{M}^+_{\lambda,\Lambda}((D^2_{\He^d}u_4)^*)=\Lambda\frac{\tilde \alpha-2}{2\rho^2}|x_H|^2[3\rho^2\tilde\alpha-(\tilde \alpha+2)]\\+\lambda\left\{\frac{\tilde \alpha-2}{2\rho^2}|x_H|^2[\tilde \alpha\rho^2-(\tilde \alpha+2)](2d-2)+3\frac{\tilde \alpha-2}{2\rho^2}|x_H|^2[\tilde \alpha\rho^2-(\tilde \alpha+2)]\right\}\\
=\frac{\tilde \alpha-2}{2\rho^2}|x_H|^2\left\{\lambda[\tilde \alpha\rho^2-(\tilde \alpha+2)](2d-2)+3\tilde \alpha\rho^2-3(\tilde \alpha+2)]+\Lambda[3\rho^2\tilde\alpha-(\tilde \alpha+2)]\right\}\\
=\frac{\tilde \alpha-2}{2\rho^2}|x_H|^2\left\{\tilde \alpha\rho^2[(2d+1)\lambda+3\Lambda] -\lambda(2d+1)(\tilde \alpha+2)-\Lambda(\tilde\alpha+2)\right\}\\
=\frac{\tilde \alpha-2}{2\rho^2}|x_H|^2\left\{\tilde \alpha\rho^2[(Q-1)\lambda+3\Lambda] -[\lambda(Q-1)+\Lambda](\tilde \alpha+2)\right\}\\
=\frac{\tilde \alpha-2}{2\rho^2}|x_H|^2\left\{[\lambda(Q-1)+\Lambda](-\tilde\alpha-2+\tilde\alpha\rho^2)+2\Lambda\tilde\alpha\rho^2\right\}\\
\leq \frac{\tilde \alpha-2}{2\rho^2}|x_H|^2\left\{-2[\lambda(Q-1)+\Lambda]+2\Lambda\tilde\alpha\right\}
=\frac{\tilde \alpha-2}{2\rho^2}|x_H|^2\left\{-2\lambda(Q-1)+2\Lambda(\tilde\alpha-1)\right\}=0\ ,
\end{multline*}
where the last equality is true in view of $\tilde{\alpha}-1=\frac{\lambda}{\Lambda}(Q-1)$. When $\rho>1$ we have
\[
f'_4(\rho)=-(2-\tilde\alpha)\rho^{1-\tilde{\alpha}} ,
\]
\[
f''_4(\rho)=-(2-\tilde\alpha)(1-\tilde\alpha)\rho^{-\tilde \alpha} ,
\]
and the eigenvalues are
\[
e_4=|D_{\He^d}\rho|^2f''_4(\rho)=-\frac{|x_H|^2(2-\tilde \alpha)(1-\tilde \alpha)}{\rho^{\tilde \alpha+2}} ,
\]
\[
e_5=3|D_{\He^d}\rho|^2\frac{f'_4(\rho)}{\rho}=-3\frac{|x_H|^2(2-\tilde \alpha)}{\rho^{\tilde \alpha+2}} ,
\]
and
\[
e_6=|D_{\He^d}\rho|^2\frac{f'_4(\rho)}{\rho}=-\frac{|x_H|^2(2-\tilde \alpha)}{\rho^{\tilde \alpha+2}}
\]
with multiplicity $2d-2$. Therefore, for $\rho\geq1$, we have
\[
\mathcal{M}^+_{\lambda,\Lambda}((D^2_{\He^d}u_4)^*)=\frac{|x_H|^2(2-\tilde \alpha)}{\rho^{\tilde \alpha+2}}\left[\Lambda(1-\tilde\alpha)+\lambda(Q-1)\right]=0\ ,
\]
Similarly, $v_4=-u_4$ yields a counterexample for the corresponding property of the miminal operator.
\end{cex}
\smallskip
{ Next we discuss the optimality of our Theorem \ref{Cor1H} in the case without lower order terms, i.e., $H_i=H_s=0$. Then the condition \eqref{condcor1} becomes $Q\leq\frac{\lambda}{\Lambda}+1$, which is not satisfied in  the Heisenberg group because $Q\geq4$. This  is consistent with the failure of the Liouville property for sub- and supersolutions of the Heisenberg Laplacian observed before. 
The next 
example shows that the Liouville property fails also for supersolutions bounded from below of $\mathcal{M}^+_{\lambda,\Lambda}((D^2_{\He^d}u)^*)=0$ and subsolutions bounded from above of  $\mathcal{M}^-_{\lambda,\Lambda}((D^2_{\He^d}u)^*)=0$, for all $\lambda, \Lambda$, and $d$. Therefore, we conclude that the presence of suitable lower order terms in Theorem \ref{Cor1H} is necessary 
for the Liouville property. 

}
%
\begin{cex} { Set $\tilde \beta:=\frac{\Lambda}{\lambda}(Q-1)+1$. Note that $\tilde \beta > 2$ because $Q\geq 4>\frac{\lambda}{\Lambda}+1$.}
In the same way as in Counterexample \ref{cexalphatilde}, one can verify that the function
\[
u_5(x)=\begin{cases}
\frac18[\tilde \beta(\tilde \beta-2)\rho^4-2(\tilde \beta^2-4)\rho^2+\tilde \beta(\tilde \beta+2)]&\text{ if }\rho<1\ ,\\
\frac{1}{\rho^{\tilde \beta-2}}&\text{ if }\rho\geq1\ .
\end{cases}
\]
is a bounded, nonconstant,  classical supersolution to $\mathcal{M}^+_{\lambda,\Lambda}((D^2_{\He^d}u_5)^*)=0$. 
\end{cex}

{
\subsection{Equations with Heisenberg Hessian and Euclidean gradient.}\label{sec;euc-grad}

Here we consider equations of the form \eqref{1L-}, i.e., \eqref{non-h-grad}, namely, 
\begin{equation}\label{non-h-grad-2}
G(x,u, Du, (D_{
\He^d}^2u)^*)=0\text{ in }\R^d\ ,
\end{equation}
with $G : \R^{2d+1}\times\R\times\R^{2d+1}\times\mathcal S_{2d} \to\R$, so they involve the Heisenberg Hessian $D_{
\He^d}^2u$ and the 
 Euclidean gradient $Du$. As at the end of Section \ref{fully} we assume $G$ is uniformly subelliptic and its first order part is bounded from below by a concave Hamiltonian $H_i$ or from above by a convex one $H_s$. 
Then Corollary \ref{eu-grad} gives one of the Liouville properties if we find a suitable super- or subsolution out of a big ball. The next result gives an explicit sufficient condition 
saying that the vector fields $b^\al$ in the drift part of $H_i$ and $H_s$ point toward the origin for $|x|$ large enough, as in the Ornstein-Uhlenbeck operators. It involves  the homogeneous norm $\rho$ of the Heisenberg group defined by \eqref{homoH}.%
\begin{cor}\label{cor_euc1}
Assume that the operator $G$ { satisfies \eqref{US}, 
where
$\X=\{X_1,...,X_{2d}\}$}  are the Heisenberg vector fields,  
and \eqref{blip} and  \eqref{cass} hold.  
Suppose there exist $\g_1,\dots, \g_{2d+1}\in\R$ with $\min_i\g_i=\g_o>0$ and such that 
\begin{equation}\label{OUtype}
\sup_\al b^\al(x)\cdot D\rho(x) \leq -\sum_{i=1}^{2d+1}\g_ix_i\partial_i\rho + o\left(\frac1{\rho^3} \right)\qquad\text{ as }\rho\to \infty.
\end{equation}
\noindent {\upshape(A)} Assume 
 \eqref{G>Hp} and $u\in\USC(\R^{2d+1})$ is a subsolution of \eqref{non-h-grad-2} satisfying  { \eqref{grow+}.}
 If either $c^\alpha(x)\equiv0$ or $u\geq0$, then $u$ is constant.

\noindent {\upshape(B)} Assume 
 \eqref{G<Hp} and $v\in\LSC(\R^{2d+1})$ is a supersolution of \eqref{non-h-grad-2} satisfying  { \eqref{grow-}.}
 If either $c^\alpha(x)\equiv0$ or $v\leq0$, then $v$ is constant.
\end{cor}
\begin{proof}
We check that $w=\log\rho$ is a supersolution. Let $C_1:= \Lambda(2d+1)-\lambda > 0$. As in the proof of Theorem \ref{Cor1H}, $w$ is a supersolution at all points where
\begin{equation}\label{tobecheck}
-C_1\frac{|x_H|^2}{\rho^4}+\inf_{\alpha\in A}\left\{c^\alpha(x)\log\rho-b^\alpha(x)\cdot\frac{D\rho}{\rho}\right\}\geq0\ .
\end{equation}
Since $D\rho = (2|x_H|^2 x_H, x_{2d+1})/(2\rho^3)$, we get  from \eqref{OUtype}  that the left hand side is larger than 
\begin{multline*}
-C_1\frac{|x_H|^2}{\rho^4}+\frac 1{2\rho^4}\left( 2\sum_{i=1}^{2d}\g_ix_i^2|x_H|^2+ \g_{2d+1}x_{2d+1}^2 + o(1)\right)\geq
\\
\frac 1{\rho^4}\left( |x_H|^2 (\g_o|x_H|^2 - C_1) +\frac{\g_o}2 x_{2d+1}^2+ o(1)\right) \geq 0 \,,
\end{multline*}
for $\rho$ large enough, by taking either $|x_H|^2 > C_1/\g_o$, or $|x_H|^2 \leq C_1/\g_o$ and $x_{2d+1}^2>2C_1^2/\g_o^2$. 
\end{proof}
The last result is based on a condition of positivity of the coefficients $c^\al$ at infinity similar to 
 Example \ref{schro}.
\begin{cor}\label{cor_euc2}
In the assumptions of Corollary \ref{cor_euc1} replace \eqref{OUtype} with 
\begin{equation}\label{c_a}
\liminf_{|x|\to\infty} \inf_{\alpha\in A}c^\alpha(x)
\log\rho(x)>0 
 \,,
\end{equation}
and either
\begin{equation}\label{sign}
\limsup_{|x|\to\infty} \sup_{\alpha\in A}b^\alpha(x)\cdot D\rho(x)\leq 0 \,,
\end{equation}
or 
\begin{equation}\label{order}
 \sup_{\alpha\in A}|b^\alpha(x)|  = o(\rho) \qquad\text{ as }\rho\to \infty.
\end{equation}
Then the  conclusions of Corollary \ref{cor_euc1} hold true.
\end{cor}
\begin{proof}
We check again the inequality \eqref{tobecheck}. Condition \eqref{sign} implies that 
$-b^\alpha(x)\cdot
{D\rho}/{\rho}\geq o(1)$ as $\rho\to \infty$ uniformly in $\al$, and the same occurs under \eqref{order} because $D\rho = O(1)$. Also ${|x_H|^2}/{\rho^4}\leq 1/\rho^2 =o(1)$. Then condition \eqref{c_a} implies  \eqref{tobecheck} for $|x|$ large enough.
\end{proof}
\begin{rem}
Corollary \ref{cor_euc1} generalizes to fully nonlinear equations the Liouville properties for linear Ornstein-Uhlenbeck operators with Heisenberg sub-laplacian proved in \cite{MMT}. 

The condition \eqref{c_a} in Corollary \ref{cor_euc2} obviously holds if $c^\alpha(x)\geq c_o>0$ for $|x|$ large enough, and in such case the condition \eqref{order} can be weakened to $\sup_{\alpha\in A}|b^\alpha(x)|  = o(\rho\log \rho)$.
\end{rem}
}


\begin{thebibliography}{10}

\bibitem{Pina}
T.~Adamowicz, A.~Kijowski, A.~Pinamonti, and B.~Warhurst.
\newblock A variational approach to the asymptotic mean-value property for the
  $p$-laplacian on {C}arnot groups.
\newblock 
{\em Nonlinear Anal.}, \newblock to appear 2020.

\bibitem{AW}
T.~Adamowicz and B.~Warhurst.
\newblock Three-spheres theorems for subelliptic quasilinear equations in
  {C}arnot groups of {H}eisenberg-type.
\newblock {\em Proc. Amer. Math. Soc.}, 144(10):4291--4302, 2016.

\bibitem{ArmstrongSirakovSharp}
S.~N. Armstrong and B.~Sirakov.
\newblock Sharp {L}iouville results for fully nonlinear equations with
  power-growth nonlinearities.
\newblock {\em Ann. Sc. Norm. Super. Pisa Cl. Sci. (5)}, 10(3):711--728, 2011.

\bibitem{ArmstrongCPAM}
S.~N. Armstrong, B.~Sirakov, and C.~K. Smart.
\newblock Fundamental solutions of homogeneous fully nonlinear elliptic
  equations.
\newblock {\em Comm. Pure Appl. Math.}, 64(6):737--777, 2011.

\bibitem{BC}
M.~Bardi and A.~Cesaroni.
\newblock Liouville properties and critical value of fully nonlinear elliptic
  operators.
\newblock {\em J. Differential Equations}, 261(7):3775--3799, 2016.

\bibitem{BCManca}
M.~Bardi, A.~Cesaroni, and L.~Manca.
\newblock Convergence by viscosity methods in multiscale financial models with
  stochastic volatility.
\newblock {\em SIAM J. Financial Math.}, 1(1):230--265, 2010.

\bibitem{BG_lio2}
M.~Bardi and A.~Goffi.
\newblock Liouville results for fully nonlinear degenerate {PDE}s modeled on
  {H}\"ormander vector fields. {II}: {C}arnot groups and {G}rushin geometries.
\newblock forthcoming.

\bibitem{BG}
M.~Bardi and A.~Goffi.
\newblock New strong maximum and comparison principles for fully nonlinear
  degenerate elliptic {PDE}s.
\newblock {\em Calc. Var. Partial Differential Equations}, 58(6):Art. 184, 20,
  2019.

\bibitem{Biri}
I.~Birindelli.
\newblock Superharmonic functions in the {H}eisenberg group: estimates and
  {L}iouville theorems.
\newblock {\em NoDEA Nonlinear Differential Equations Appl.}, 10(2):171--185,
  2003.

\bibitem{BCDC}
I.~Birindelli, I.~Capuzzo~Dolcetta, and A.~Cutr\`\i.
\newblock Liouville theorems for semilinear equations on the {H}eisenberg
  group.
\newblock {\em Ann. Inst. H. Poincar\'{e} Anal. Non Lin\'{e}aire},
  14(3):295--308, 1997.

\bibitem{BGL}
I.~Birindelli, G.~Galise, and F.~Leoni.
\newblock Liouville theorems for a family of very degenerate elliptic nonlinear
  operators.
\newblock {\em Nonlinear Anal.}, 161:198--211, 2017.

\bibitem{BLU}
A.~Bonfiglioli, E.~Lanconelli, and F.~Uguzzoni.
\newblock {\em Stratified {L}ie groups and potential theory for their
  sub-{L}aplacians}.
\newblock Springer Monographs in Mathematics. Springer, Berlin, 2007.

\bibitem{BFP}
S.~Bordoni, R.~Filippucci, and P.~Pucci.
\newblock Nonlinear elliptic inequalities with gradient terms on the
  {H}eisenberg group.
\newblock {\em Nonlinear Anal.}, 121:262--279, 2015.

\bibitem{BMagliaro}
L.~Brandolini and M.~Magliaro.
\newblock Liouville type results and a maximum principle for non-linear
  differential operators on the {H}eisenberg group.
\newblock {\em J. Math. Anal. Appl.}, 415(2):686--712, 2014.

\bibitem{CC}
L.~A. Caffarelli and X.~Cabr\'{e}.
\newblock {\em Fully nonlinear elliptic equations}, volume~43 of {\em American
  Mathematical Society Colloquium Publications}.
\newblock American Mathematical Society, Providence, RI, 1995.

\bibitem{Italosurvey}
I.~Capuzzo~Dolcetta.
\newblock Liouville theorems and a priori estimates for elliptic semilinear
  equations.
\newblock {\em Rend. Sem. Mat. Fis. Milano}, 68:1--18 (2001), 1998.

\bibitem{ICDCsub}
I.~Capuzzo~Dolcetta and A.~Cutr\`\i.
\newblock On the {L}iouville property for sub-{L}aplacians.
\newblock {\em Ann. Scuola Norm. Sup. Pisa Cl. Sci. (4)}, 25(1-2):239--256
  (1998), 1997.
\newblock Dedicated to Ennio De Giorgi.

\bibitem{CDC2003}
I.~Capuzzo~Dolcetta and A.~Cutr\`\i.
\newblock Hadamard and {L}iouville type results for fully nonlinear partial
  differential inequalities.
\newblock {\em Commun. Contemp. Math.}, 5(3):435--448, 2003.

\bibitem{CFelmer}
H.~Chen and P.~Felmer.
\newblock On {L}iouville type theorems for fully nonlinear elliptic equations
  with gradient term.
\newblock {\em J. Differential Equations}, 255(8):2167--2195, 2013.

\bibitem{CLeoni}
A.~Cutr\`i and F.~Leoni.
\newblock On the {L}iouville property for fully nonlinear equations.
\newblock {\em Ann. Inst. H. Poincar\'{e} Anal. Non Lin\'{e}aire},
  17(2):219--245, 2000.

\bibitem{CTchou}
A.~Cutr\`i and N.~Tchou.
\newblock Barrier functions for {P}ucci-{H}eisenberg operators and
  applications.
\newblock {\em Int. J. Dyn. Syst. Differ. Equ.}, 1(2):117--131, 2007.

\bibitem{FP}
C.~Fefferman and D.~H. Phong.
\newblock Subelliptic eigenvalue problems.
\newblock In {\em Conference on harmonic analysis in honor of {A}ntoni
  {Z}ygmund, {V}ol. {I}, {II} ({C}hicago, {I}ll., 1981)}, Wadsworth Math. Ser.,
  pages 590--606. Wadsworth, Belmont, CA, 1983.

\bibitem{FR}
E.~Feleqi and F.~Rampazzo.
\newblock Iterated {L}ie brackets for nonsmooth vector fields.
\newblock {\em NoDEA Nonlinear Differential Equations Appl.}, 24(6):Art. 61,
  43, 2017.

\bibitem{Folland}
G.~B. Folland.
\newblock Subelliptic estimates and function spaces on nilpotent {L}ie groups.
\newblock {\em Ark. Mat.}, 13(2):161--207, 1975.

\bibitem{GS_cpde}
B.~Gidas and J.~Spruck.
\newblock A priori bounds for positive solutions of nonlinear elliptic
  equations.
\newblock {\em Comm. Partial Differential Equations}, 6(8):883--901, 1981.

\bibitem{GT}
D.~Gilbarg and N.~S. Trudinger.
\newblock {\em Elliptic partial differential equations of second order}, volume
  224 of {\em Grundlehren der Mathematischen Wissenschaften [Fundamental
  Principles of Mathematical Sciences]}.
\newblock Springer-Verlag, Berlin, second edition, 1983.

\bibitem{Goffi}
A.~Goffi.
\newblock Some new {L}iouville-type results for fully nonlinear {PDE}s on the
  {H}eisenberg group.
\newblock {\em Nonlinear Anal.}, 
\newblock to appear 2020.

\bibitem{Gry1}
A.~Grigor'yan.
\newblock Analytic and geometric background of recurrence and non-explosion of
  the {B}rownian motion on {R}iemannian manifolds.
\newblock {\em Bull. Amer. Math. Soc. (N.S.)}, 36(2):135--249, 1999.

\bibitem{KL2}
A.~E. Kogoj and E.~Lanconelli.
\newblock Liouville theorem for {$X$}-elliptic operators.
\newblock {\em Nonlinear Anal.}, 70(8):2974--2985, 2009.

\bibitem{KL1}
A.~E. Kogoj and E.~Lanconelli.
\newblock {$L^p$}-{L}iouville theorems for invariant partial differential
  operators in {${\R}^n$}.
\newblock {\em Nonlinear Anal.}, 121:188--205, 2015.

\bibitem{LeoniJMPA}
F.~Leoni.
\newblock Explicit subsolutions and a {L}iouville theorem for fully nonlinear
  uniformly elliptic inequalities in halfspaces.
\newblock {\em J. Math. Pures Appl. (9)}, 98(5):574--590, 2012.

\bibitem{Li}
Y.~Y. Li, L.~Nguyen, and B.~Wang.
\newblock Towards a {L}iouville theorem for continuous viscosity solutions to
  fully nonlinear elliptic equations in conformal geometry.
\newblock 
{\em In ``Geometric Analysis, In Honor of Gang Tian's 60th Birthday''}, Progress in Mathematics, Birkh\"auser, p. 221-244, 2020. 
\bibitem{MMT}
P.~Mannucci, C.~Marchi, and N.~Tchou.
\newblock The ergodic problem for some subelliptic operators with unbounded
  coefficients.
\newblock {\em NoDEA Nonlinear Differential Equations Appl.}, 23(4):Art. 47,
  26, 2016.

\bibitem{MMTesaim}
P.~Mannucci, C.~Marchi, and N.~Tchou.
\newblock Singular perturbations for a subelliptic operator.
\newblock {\em ESAIM Control Optim. Calc. Var.}, 24(4):1429--1451, 2018.

\bibitem{NV}
N.~Nadirashvili, V.~Tkachev, and S.~Vl\u{a}du\c{t}.
\newblock {\em Nonlinear elliptic equations and nonassociative algebras},
  volume 200 of {\em Mathematical Surveys and Monographs}.
\newblock American Mathematical Society, Providence, RI, 2014.

\bibitem{PP}
A.~Porretta and E.~Priola.
\newblock Global {L}ipschitz regularizing effects for linear and nonlinear
  parabolic equations.
\newblock {\em J. Math. Pures Appl. (9)}, 100(5):633--686, 2013.

\bibitem{PZ}
E.~Priola and J.~Zabczyk.
\newblock Liouville theorems for non-local operators.
\newblock {\em J. Funct. Anal.}, 216(2):455--490, 2004.

\bibitem{PW}
M.~H. Protter and H.~F. Weinberger.
\newblock {\em Maximum principles in differential equations}.
\newblock Springer-Verlag, New York, 1984.
\newblock Corrected reprint of the 1967 original.

\bibitem{Pucci66}
C.~Pucci.
\newblock Operatori ellittici estremanti.
\newblock {\em Ann. Mat. Pura Appl. (4)}, 72:141--170, 1966.

\bibitem{Rossi}
L.~Rossi.
\newblock Non-existence of positive solutions of fully nonlinear elliptic
  equations in unbounded domains.
\newblock {\em Commun. Pure Appl. Anal.}, 7(1):125--141, 2008.


\end{thebibliography}

%
\end{document}